\documentclass[12pt]{article}

\usepackage{amscd,amssymb}
\usepackage{amsfonts}
\usepackage{theorem,url}
\usepackage[leqno]{amsmath}
\usepackage{enumerate}
\usepackage{indentfirst, color}

\newenvironment{proof}[1][Proof]{\textbf{#1.} }{\hfill $\square$}

\newtheorem{lmm}{Lemma}
\newtheorem{thrm}{Theorem}
\newtheorem{dfntn}{Definition}
\newtheorem{rmrk}{Remark}
\newtheorem{xmpl}{Example}
\newtheorem{prpstn}{Proposition}

\newenvironment{enonce}[2][]%
{\noindent  \textbf{#2}} %
{\hfill $\diamond$}

\newcommand{\bord}{\partial \mathcal{S}}
\newcommand{\tOm}{\widetilde{\Omega}}
\newcommand{\R}{\mathbb{R}}
\newcommand{\N}{\mathbb{N}}
\newcommand{\Prb}{\mathbb{P}}
\newcommand{\E}{\mathbb{E}}
\newcommand{\bD}{\mathbb{D}}
\newcommand{\bM}{\mathbb{M}}
\newcommand{\bH}{\mathbb{H}}
\newcommand{\bL}{\mathbb{L}}
\newcommand{\bS}{\mathbb{S}}
\newcommand{\bF}{\mathbb{F}}
\newcommand{\tri}{\mathcal{F}}
\newcommand{\cP}{\mathcal{P}}
\newcommand{\cR}{\mathcal{R}}

\newcommand{\cD}{\mathcal{D}}

\newcommand{\cS}{\mathcal{S}}

\newcommand{\cE}{\mathcal{E}}

\newcommand{\cM}{\mathcal{M}}
\newcommand{\tP}{\widetilde{\mathcal{P}}}
\newcommand{\eps}{\varepsilon}

\newcommand{\al}{\alpha}

\newcommand{\la}{\lambda}
\newcommand{\ind}{\mathbf{1}}

\newcommand{\tr}{\text{Trace}}

\newcommand{\tmu}{\widetilde{\mu}}

\newcommand{\ope}{\mathcal{L}}

\newcommand{\supp}{\text{supp}}
\newcommand{\cad}{c\`adl\`ag }

\title{Limit behaviour of the minimal solution of a BSDE with jumps and with singular terminal condition.}
\author{A. Popier \\*
LUNAM Universit\'e, Universit\'e du Maine, \\* Laboratoire Manceau de Math\'ematiques,\\*
 Avenue O. Messiaen, 72085 Le Mans cedex 9\\* France}

\begin{document}

\maketitle

\begin{abstract} 
We study the behaviour at the terminal time $T$ of the minimal solution of a backward stochastic differential equation when the terminal data can take the value $+\infty$ with positive probability. In a previous paper \cite{krus:popi:15}, we have proved existence of this minimal solution (in a weak sense) in a quite general setting. But two questions arise in this context and were still open: is the solution right continuous with left limits on $[0,T]$? In other words does the solution have a left limit at time $T$? The second question is: is this limit equal to the terminal condition? In this paper, under additional conditions on the generator and the terminal condition, we give a positive answer to these two questions. 
\end{abstract}
%
%%\begin{resume} 
%Nous \'etudions le comportement \`a l'instant final $T$ de la solution minimale d'une \'equation diff\'erentielle stochastiques r\'etrograde, pour laquelle la condition terminale peut \^etre \'egal \`a $+\infty$ avec une probabilit\'e strictement positive. Dans notre papier pr\'ec\'edent \cite{krus:popi:15}, nous avons prouv\'e l'existence de cette solution minimale (au sens faible) dans un cadre tr\`es large. N\'eanmoins deux questions naturelles restaient ouvertes : cette solution est-elle continue \`a droite avec une limite \`a gauche en tout point ? Autrement dit, a-t-elle une limite \`a gauche en $T$ ? Et deuxi\`emement cette limite est-elle \'egale \`a la condition terminale ? Dans ce papier, sous des hypoth\`eses suppl\'ementaires sur le g\'en\'erateur et la condition terminale, nous donnons une r\'eponse positive \`a ces deux interrogations.
%\end{resume}
%
\noindent {\bf AMS class:} 60G99, 60H99, 60J15.

\section*{Introduction}
%---------------

Backward stochastic differential equations (BSDE) were introduced in \cite{bism:73} in the linear case and extended in the non linear case in \cite{pard:peng:90}. Since then a huge literature has been developed on this topic and on their applications (see for example \cite{delo:13} or \cite{pard:rasc:14} and the references therein). In this paper we consider a filtered probability space $(\Omega,\tri, \bF,\Prb)$ with a complete and right-continuous filtration $\bF=\{\tri_t, \ t \geq 0\}$. We assume that this space supports a Brownian motion $W$ and a Poisson random measure $\mu$ with intensity with intensity $\lambda(de)dt$ on a space $E$. $\tilde \mu$ denotes the compensated related martingale. We consider the following BSDE:
\begin{equation}\label{eq:BSDE_jumps}
Y_t=\xi + \int_t^T f(s,Y_s,Z_s,U_s) ds - \int_t^T Z_s dW_s - \int_t^T\int_E U_s(e) \tilde{\mu}(ds,de) - \int_t^T dM_s
\end{equation}
where $f$ is the generator and $\xi$ is the terminal condition. The solution is the quadruplet $(Y,Z,U,M)$. Since no particular assumption is made on the underlying filtration, there is the additional martingale part $M$, orthogonal to $W$ and $\mu$. It is already established that such a BSDE has a unique solution when the terminal condition $\xi$ belongs to $L^p(\Omega,\mathcal{F}_T,\mathbb{P})$, $p >1$ (see among others \cite{barl:buck:pard:97}, \cite{delo:13} or \cite{krus:popi:14}). 

When the terminal condition $\xi$ satisfies
\begin{equation} \label{eq:sing_term_cond}
\Prb(\xi = +\infty)  >0
\end{equation}
we called the BSDE {\it singular}. This singular case has been studied in \cite{popi:06} when the filtration is generated by the Brownian motion (no jump, no additional noise, i.e. $U=M=0$) and for the particular generator $f(t,y,z,u)=f(y)=-y|y|^q$. Recently singular BSDE were used to solve a particular stochastic control problem with application to portfolio management (see \cite{anki:jean:krus:13} or \cite{grae:hors:qiu:13}). In this framework, the intensity $\lambda$ is finite and the generator does not depend on $z$ and has the following form:
\begin{equation} \label{eq:example_control}
f(t,y,u) = - \frac{y|y|^{q}}{q \alpha_t^{q}} - \widehat \beta(t,y,u) + \gamma_t.
\end{equation}
where the function $\widehat \beta$ is given by
\begin{equation}\label{eq:generater_control_psi}
\widehat \beta(t,y,u) = \int_E (y+ u(e)) \left( 1- \frac{\beta_t(e)}{\left((y+ u(e))^{q}+ \beta_t(e)^{q} \right)^{1/q}}  \right) \ind_{y+ u(e) \geq 0} \ \lambda(de)
\end{equation}
and where $\alpha$, $\beta$ and $\gamma$ are positive processes. The minimal solution $(Y,Z,U,M)$ (provided it exists) gives the value function of the following control problem: minimize\footnote{with the convention $0.\infty = 0$.}
\begin{equation}\label{eq:control_pb_intro}
\E \left[  \int_t^T \left( \alpha_s |\eta_s|^p + \gamma_s |X_s|^p + \int_E \beta_s(e) |\zeta_s(e)|^p \lambda(de) \right) ds  + \xi  |X_T|^p \bigg| \tri_t \right]
\end{equation}
over all progressively measurable processes $X$ that satisfy the dynamics
\begin{equation*}
X_s =x +\int_t^s \eta_u du +\int_t^s \int_E \zeta_u(e) \mu(de,du)
\end{equation*}
and the terminal state constraint 
\begin{equation*}
X_T \ind_{\xi=\infty}= 0.
\end{equation*}
$p$ is the H\"older conjugate of $1+q$. For the financial point of view, the set $\{\xi=+\infty\}$ is a specification of a set of market scenarios where liquidation is mandatory. The value function is equal to $|x|^p Y_t$ and the optimal state process $X^*$ can be computed directly with $Y$ and $U$. Note that the martingale part of the solution $(Z,M)$ is not employed in the computation of the optimal state process. Thus the control problem can be completely solved provided the BSDE has a minimal solution (see Section 2 and Theorem 4 in \cite{krus:popi:15} for more details on the control problem).

In \cite{krus:popi:15}, under some technical sufficient assumptions on $f$ (Conditions \textbf{(A)} below), it is proved that the BSDE \eqref{eq:BSDE_jumps} with singular terminal condition \eqref{eq:sing_term_cond} has a minimal super-solution $(Y,Z,U,M)$ such that a.s.
\begin{equation} \label{eq:term_cond_super_sol}
\liminf_{t\to T} Y_t \geq \xi.
\end{equation}
The main requirement is that $f$ decreases w.r.t. $y$ at least as a polynomial function (almost like $-y^{1+q}$, $q>0$), when $y$ is large. The main difficulty is to obtain some a priori estimate, which states that $Y_t$ is bounded from above for any $t < T$ by a finite process (Inequality \eqref{eq:a_priori_estimate}). The construction of $(Y,Z,U,M)$ without Condition \eqref{eq:term_cond_super_sol} can be made if the filtration $\bF$ is only complete and right-continuous. 

In the classical setting ($\xi \in L^p(\Omega)$), $Y$ has a limit as $t$ increases to $T$ since the process is solution of the BSDE \eqref{eq:BSDE_jumps} and thus is c\`adl\`ag\footnote{French acronym for right continuous with left limits.}. Moreover this limit is equal to $\xi$ a.s. if the filtration $\bF$ is left-continuous at time $T$. Indeed we need to avoid a jump at time $T$ of the orthogonal martingale $M$ (see the later discussion in Section \ref{sect:setting}). Hence in the singular case the behaviour \eqref{eq:term_cond_super_sol} of the super-solution $Y$ at time $T$ is obtained under this additional requirement on the filtration $\bF$. For the related control problem \eqref{eq:control_pb_intro}, this weak behaviour \eqref{eq:term_cond_super_sol} at time $T$ of the minimal process $Y$ is sufficient to obtain the optimal control and the value function. Nevertheless two natural questions arise here:
\begin{itemize}
\item Can we expect that the left limit at time $T$ of the minimal solution $Y$ exists ? In other words is $Y$ c\`adl\`ag on $[0,T]$ ? 
\item Can the inequality \eqref{eq:term_cond_super_sol} be an equality if the filtration is quasi left-continuous ? 
\end{itemize}
The aim of this paper is to give an (at least partial) answer to these two questions. 

\subsubsection*{Related literature}
%----------------

As far as we know, there are only two works on this topic: \cite{popi:06} in the Brownian setting and the third chapter of the PhD thesis of Piozin in the Brownian-Poisson setting, both for the particular generator $f(y) = -y|y|^q$. 

The study of this question in the Brownian setting was firstly made in \cite{popi:06}. The proof was decomposed into two parts: firstly the existence of the limit, secondly the equality
\begin{equation} \label{eq:equality_liminf_T}
\liminf_{t\to T} Y_t = \xi.
\end{equation}
For the existence of the limit no additional assumption was required. But the particular structure of $f$ was very important. For the second part to prove the equality \eqref{eq:equality_liminf_T}, a localization procedure was used, working only in the Markovian framework. To be more precised the terminal condition $\xi$ is equal to $h(X_T)$ where $h$ is a function from $\R^d$ to $[0,+\infty]$ with a closed non empty singular set $\cS=\{ h = +\infty\}$ and $X$ is a forward diffusion given as the solution of a SDE. For $q \leq 2$, Malliavin's calculus and an integration by parts were used, in order to transfert the control of $Z_t = D_t Y_t$ to the control of $Y_t$. Since the density of the process $X$ appears in this integration by parts, additional conditions were required on the diffusion $X$ (especially uniform ellipticity). For $q > 2$, no particular hypothesis was imposed on $X$ since we had a suitable estimate on $Z$. In some sense if the non linearity $q$ is not large enough, it should be compensated by more regularity on $X$. 

In the third chapter of the PhD thesis of Piozin, we have studied the Brownian-Poisson case, again with the generator $f(y) = -y|y|^q$. The existence of the limit is treated exactly as in \cite{popi:06}. For the proof of the equality \eqref{eq:equality_liminf_T} we worked in the Markovian setting: $\xi = h(X_T)$, and we assumed that $q>2$ since we had an a priori estimate on $Z$ and $U$. Moreover the localization procedure involves the infinitesimal generator of $X$ which contains a non local part due to the jumps in the forward SDE. In order to have suitable estimates on this part, we needed other technical assumptions on the jumps of the solution $X$ and the singular set $\cS = \{\xi = +\infty\}$.

\subsubsection*{Contributions and decomposition of the paper}
%----------------

In this paper we want to broaden these results in two directions: more general generator $f$ and no restriction on the filtration. We are not able to give a general result for any generator $f$ and any final value $\xi$. We just have sufficient conditions on the generator $f$ and the terminal value $\xi$ as in \cite{popi:06}. Unfortunately there is still a gap between existence conditions of the minimal super-solution and ``continuity at time $T$'' assumptions. \begin{itemize}
\item We can prove the existence of this limit only if the generator $f$ has some specific structure (Theorem \ref{thm:exists_limit}). Our main requirement is that the growth of $f$ w.r.t. $y$ could be ``controlled'' uniformly w.r.t. $z$ and $u$. The filtration $\bF$ satisfies the usual assumptions: complete and right-continuous.
\item The proof of the equality \eqref{eq:equality_liminf_T} is a different problem. First we will assume that the filtration $\bF$ is such that a martingale cannot have a jump at time $T$. This property holds if $\bF$ is quasi left-continuous, which is a common assumption (see the discussion in Section \ref{ssect:filtration}). The setting will be half Markovian: $\xi=h(X_T)$. But the generator is not supposed to depend only on the forward process $X$. Hence the setting is not completely Markovian. As explained before we will assume that $q$ is large enough ($q>2$ if $f(y) = -y|y|^q$). The complete result is Theorem \ref{thm:equality}.
\end{itemize}
Our method here is very closed to \cite{popi:06}. Nevertheless some added difficulties appear here; let us briefly explain them. For the existence of the left limit of $Y$, if the generator $f$ does not depend on $U$, the proof is almost the same as in \cite{popi:06}. In other words the presence of $U$ in the generator is the main trouble and the reason why additional hypotheses are made on $f$. To obtain the desired equality \eqref{eq:equality_liminf_T}, as in \cite{popi:06}, we will prove a stronger a priori estimate on $Z$ and $U$ (extension of a known result in the Brownian setting). This point is rather technical but not surprising. Next we will restrict ourselves to the case where $\xi$ depends only on the terminal value of a forward diffusion $X$: $\xi = \Phi(X_T)$. The complexity does not come from the filtration or the martingale part. Indeed with a left-continuity condition at time $T$ on the filtration $\bF$, the martingale $M$ will not have a jump at time $T$ and the martingale part of the BSDE will be killed by taking the expectation. But the novelty is the non local part $\mathcal{I}$ of the infinitesimal generator of $X$ (given by \eqref{eq:non_local_gene}). This term comes from the jumps of $X$ generated by the Poisson random measure. We will be able to have a suitable estimate on this term under a technical condition (denoted by {\bf (E)}). This assumption connects the jumps of $X$ with the singularity set $\mathcal{S} = \{\Phi = +\infty\}$ and is original and unusual. 

The paper is organized as follows. In the first part we give the precise mathematical framework and we recall the known result: existence of the minimal solution of the BSDE \eqref{eq:BSDE_jumps}. We also discuss the required conditions on the filtration $\bF$. Then in Section \ref{sect:exist_limit} we prove existence of a left limit, that is $Y$ is c\`adl\`ag on $[0,T]$ (Theorem \ref{thm:exists_limit} and its proof) under Condition \textbf{(B)}. In the last section we want to prove Equality \eqref{eq:equality_liminf_T}. First we will show that the generator cannot be singular. We complete the results of \cite{krus:popi:15} with an a priori estimate of the coefficients $Z$ and $U$. And finally under other technical assumptions we prove Equality \eqref{eq:equality_liminf_T} (Theorem \ref{thm:equality}). Since the conditions in the second and third sections are not exactly the same, we gather all conditions in the short last section in order to obtain the continuity of $Y$ at time $T$
\begin{equation} \label{eq:continuity_T}
\lim_{t\to T} Y_t = \xi.
\end{equation}
Along the paper we will always consider Example \ref{ex:control_ex} (generator $f$ given by \eqref{eq:example_control} and \eqref{eq:generater_control_psi}), with the important particular subcases: 
\begin{itemize}
\item Example \ref{ex:toy_ex}: $f(y)=-y|y|^q$ (toy example). 
\item Example \ref{ex:power_sing}: $\displaystyle f(t,y)=-(T-t)^\varsigma y|y|^q + \frac{1}{(T-t)^\varpi}$, with real numbers $\varsigma$ and $\varpi$. 
\end{itemize}

The open questions ($q$ small, more general generator $f$, non Markovian setting, etc.) are left for further developments. 

\section{Setting, known results, filtration} \label{sect:setting}
%-----------------------

We consider a filtered probability space $(\Omega,\tri,\Prb,\bF = (\tri_t)_{t\geq 0})$. The filtration is assumed to be complete and right continuous. Note that all martingales have right continuous modifications in this setting and we will always assume that we are taking the right continuous version, without any special mention. 
We also suppose sometimes that the filtration is quasi-left continuous (as in \cite{krus:popi:14} and \cite{krus:popi:15}), which means that for every sequence $(\tau_n)$ of $\bF$ predictable stopping times such that $\tau_n \nearrow \tilde \tau$ for some stopping time $\tilde \tau$ we have $\bigvee_{n\in \N}\tri_{\tau_n}=\tri_{\tilde \tau}$. This condition is unimportant for the existence and/or uniqueness of the solution of the BSDE (see \cite{bouc:poss:zhou:15} for more details on this technical point). 
 
We assume that $(\Omega,\tri,\Prb,\bF = (\tri_t)_{t\geq 0})$ supports a $d$-dimensional Brownian motion $W$ and a Poisson random measure $\mu$ with intensity $\lambda(de)dt$ on the space $E \subset \R^{d'} \setminus \{0\}$. We will denote $\cE$ the Borelian $\sigma$-field of $E$ and $\tmu$ is the compensated measure: for any $A \in \cE$ such that $\lambda(A)  < +\infty$, then $\tmu([0,t] \times A) = \mu([0,t] \times A) - t \lambda(A)$ is a martingale. The measure $\lambda$ is $\sigma$-finite on $(E,\cE)$ satisfying 
$$\int_E (1\wedge |e|^2) \lambda(de) < +\infty.$$

In this paper for a given $T\geq 0$, we denote:
\begin{itemize}
\item $\cP$: the predictable $\sigma$-field on $\Omega \times [0,T]$ and
$$\tP=\cP \otimes \cE.$$
\item On $\tOm = \Omega \times [0,T] \times E$, a function that is $\tP$-measurable, is called predictable. $G_{loc}(\mu)$ is the set of $\tP$-measurable functions $\psi$ on $\tOm$ such that for any $t \geq 0$ a.s.
$$ \int_0^t \int_E (|\psi_s(e)|^2\wedge |\psi_s(e)|) \lambda(de) < +\infty.$$
\item $\cD$ (resp. $\cD(0,T)$): the set of all predictable processes on $\R_+$ (resp. on $[0,T]$). $L^2_{loc}(W)$ is the subspace of $\cD$ such that for any $t\geq 0$ a.s.
$$\int_0^t |Z_s|^2 ds < +\infty.$$
\item $ \cM_{loc}$: the set of c\`adl\`ag local martingales orthogonal to $W$ and $\tmu$. If $M \in \cM_{loc}$ then
$$[ M, W^i]_t =0, 1\leq i \leq k \qquad [M ,\tmu(A,.)]_t = 0$$
for all $A\in \cE$.
In other words, $\E (\Delta M * \mu | \tP) = 0$, where the product $*$ denotes the integral process (see II.1.5 in \cite{jaco:shir:03}). Roughly speaking, the jumps of $M$ and $\mu$ are independent. 
\item $\cM$ is the subspace of $ \cM_{loc}$ of martingales.
\end{itemize}
We refer to \cite{jaco:shir:03} for details on random measures and stochastic integrals. On $\R^d$, $|.|$ denotes the Euclidean norm and $\R^{d\times d'}$ is identified with the space of real matrices with $d$ rows and $d'$ columns. If $z \in  \R^{d\times d'}$, we have $|z|^2 = \mbox{trace}(zz^*)$. 

Now to define the solution of our BSDE, let us introduce the following spaces for $p\geq 1$.
\begin{itemize}
\item $\bD^p(0,T)$ is the space of all adapted \cad processes $X$ such that
$$\E \left(  \sup_{t\in [0,T]} |X_t|^p \right) < +\infty.$$
For simplicity, $X_* = \sup_{t\in [0,T]} |X_t|$.
\item $\bH^p(0,T)$ is the subspace of all processes $X\in \cD(0,T)$ such that
$$\E \left[ \left( \int_0^T |X_t|^2 dt\right)^{p/2} \right] < +\infty.$$
\item $\bM^p(0,T)$ is the subspace of $\cM$ of all martingales such that
$$\E \left[ \left( [ M ]_T \right)^{p/2}\right] < +\infty.$$
\item $\bL^p_\mu(0,T) = \bL^p_{\mu}(\Omega\times (0,T)\times E)$: the set of processes $\psi \in G_{loc}(\mu)$ such that
$$\E \left[ \left(  \int_0^T \int_{E} |\psi_s(e)|^2 \mu(ds,de) \right)^{p/2} \right] < +\infty .$$
\item $\bL^p_\lambda(E)=\bL^p(E,\lambda;\R^m)$: the set of measurable functions $\psi : E \to \R^m$ such that
$$\| \psi \|^p_{\bL^p_\lambda} = \int_{E} |\psi(e)|^p \lambda(de)  < +\infty .$$
\item $\bS^p(0,T) = \bD^p(0,T) \times \bH^p(0,T) \times \bL^p_\mu(0,T) \times \bM^p(0,T)$.
\end{itemize}
If $M$ is a $\R^d$-valued martingale in $\cM$, the bracket process $[ M ]_t$ is
$$[ M ]_t = \sum_{i=1}^d [ M^i ]_t,$$
where $M^i$ is the $i$-th component of the vector $M$. 

We consider the BSDE \eqref{eq:BSDE_jumps}
\begin{equation*} 
Y_t = \xi + \int_t^T f(s,Y_s, Z_s,U_s) ds  -\int_t^T Z_s dW_s- \int_t^T\int_E U_s(e) \tmu(de,ds)- \int_t^T dM_s.
\end{equation*}
Here, the random variable $\xi$ is $\tri_T$-measurable with values in $\R$ and the generator $f : \Omega \times [0,T] \times \R \times \R^{d} \times \bL^2_\lambda(E) \to \R$ is a random function, measurable with respect to $Prog \times \mathcal{B}(\R)\times \mathcal{B}(\R^{d}) \times \mathcal{B}(\bL^2_\lambda(E))$ where $Prog$ denotes the sigma-field of progressive subsets of $\Omega \times [0,T]$. The unknowns are $(Y,Z,U,M)$ such that
\begin{itemize}
\item $Y$ is progressively measurable and \cad with values in $\R$;
\item $Z \in L^2_{loc}(W)$, with values in $\R^{d}$;
\item $U \in G_{loc}(\mu)$ with values in $\R$;
\item $M \in \cM_{loc}$ with values in $\R$.
\end{itemize}
For notational convenience we will denote $f^0_t = f(t,0,0,0)$. 

\subsubsection*{Assumptions}
%-------------

\begin{itemize}
\item $\xi$ and $f^0_t$ are non negative and $\Prb(\xi = +\infty) > 0$. $\cS$ is the set of singularity: 
$$\cS = \{\xi=+\infty \}.$$
\item The function $y\mapsto f(t,y,z,u)$ is continuous and monotone: there exists $\chi \in \R$ such that a.s. and for any $t \in [0,T]$, $z \in \R^k$ and $u \in \bL^2_\la(E)$
\begin{equation}\label{eq:f_mono} \tag{A1}
(f(t,y,z,u)-f(t,y',z,u))(y-y') \leq \chi (y-y')^2.
\end{equation}
\item For every $n> 0$ the function
\begin{equation}\label{eq:f_growth_y_t} \tag{A2}
\sup_{|y|\leq n} |f(t,y,0,0)-f^0_t| \in L^1((0,T) \times \Omega).
\end{equation}
\item $f$ is Lispchitz in $z$, uniformly w.r.t. all parameters: there exists $L > 0$ such that for any $(t,y,u)$, $z$ and $z'$: a.s.
\begin{equation}\label{eq:f_lip_z} \tag{A3}
|f(t,y,z,u)-f(t,y,z',u)| \leq L |z-z'|.
\end{equation}
\item There exists a progressively measurable process $\kappa = \kappa^{y,z,u,v} : \Omega \times \R_+ \times E \to \R$ such that
\begin{equation}\label{eq:f_jump_comp} \tag{A4}
f(t,y,z,u)-f(t,y,z,v) \leq \int_E (u(e)-v(e))  \kappa^{y,z,u,v}_t(e)  \la(de)
\end{equation}
with $\Prb \otimes Leb \otimes \la$-a.e. for any $(y,z,u,v)$, $-1 \leq \kappa^{y,z,u,v}_t(e)$
and $|\kappa^{y,z,u,v}_t(e)|\leq \vartheta(e)$ where $\vartheta \in \bL^2_\la(E)$.
\end{itemize}
Note that no assumption on $f^0$ (expect non negativity) is required. Conditions \eqref{eq:f_mono}-\eqref{eq:f_jump_comp} will ensure existence and uniqueness of the solution for a version of BSDE \eqref{eq:BSDE_jumps}, where the terminal condition $\xi$ is replaced by $\xi \wedge n$ and where the generator $f$ is replaced by $f_n = f-f^0 + (f^0\wedge n)$ for some $n>0$ (see BSDE \eqref{eq:trunc_BSDE} below). We obtain the minimal supersolution (see Theorem \ref{thm:exists_super_sol}) with singular terminal condition $\xi$ by letting the truncation $n$ tend to $\infty$. To ensure that in the limit (when $n$ goes to $\infty$) the solution component $Y$ attains the value $\infty$ on $\cS$ at time $T$ but is finite before time $T$, we suppose that 
\begin{itemize}
\item There exists a constant $q > 0$ and a positive process $a$ such that for any $y \geq 0$
\begin{equation}\label{eq:f_upper_bound} \tag{A5}
f(t,y,z,u)\leq - (a_t) y|y|^{q} + f(t,0,z,u).
\end{equation}

\end{itemize}
$p=1+\frac{1}{q}$ is the H\"older conjugate of $1+q$. Moreover, in order to derive the a priori estimate, the following assumptions will hold.
\begin{itemize}
\item There exists some $\ell > 1$ such that
\begin{equation}\label{eq:alpha_gamma} \tag{A6}
\E \int_0^T \left[ \left( \frac{1}{qa_s}\right)^{1/q}+ (T-s)^p f^0_s\right]^{\ell} ds < +\infty.
\end{equation}
\item There exists $k > \max (2, \ell/(\ell-1))$ such that 
\begin{equation}\label{eq:f_growth_psi} \tag{A7}
\int_E |\vartheta(e)|^{k} \la(de) < +\infty.
\end{equation}
\end{itemize}
\begin{dfntn}
The generator $f$ satisfies Conditions {\rm \textbf{(A)}} if all assumptions \eqref{eq:f_mono}--\eqref{eq:f_growth_psi} hold.
\end{dfntn}

\begin{rmrk}[on Assumptions {\rm \textbf{(A)}}]
$\ $ 
\begin{enumerate}
\item By very classical arguments we can suppose w.l.o.g.\ that $\chi=0$ in \eqref{eq:f_mono}. \textbf{In the rest of the paper, we assume that $\chi=0$.}
\item Assumptions \eqref{eq:f_growth_y_t} and \eqref{eq:f_upper_bound} imply that the process $a$ must be in $L^1((0,T)\times \Omega)$. Indeed from \eqref{eq:f_upper_bound} with $z=u=0$ and $y=1$, we obtain
$$f(t,1,0,0)- f^0_t \leq - (a_t) \Rightarrow a_t \leq  |f(t,1,0,0)- f^0_t|.$$
Condition \eqref{eq:f_growth_y_t} leads to the integrability of $a$.
\item The fourth condition \eqref{eq:f_jump_comp} implies that $f$ is Lipschitz continuous w.r.t. $u$ uniformly in $\omega$, $t$, $y$ and $z$: 
$$|f(t,y,z,u)-f(t,y,z,v)| \leq \|\vartheta\|_{L^2_\la} \|u-v\|_{L^2_\la}  = L \|u-v\|_{L^2_\la}.$$
This assumption \eqref{eq:f_jump_comp} is used to compare two solutions of the BSDE \eqref{eq:BSDE_jumps} with different terminal conditions (see Theorem 4.1 and Assumption 4.1 in \cite{quen:sule:13} or Proposition 4 in \cite{krus:popi:14}). 
\item The generator $f$ can be also ``singular'' at time $T$ provided Assumption \eqref{eq:alpha_gamma} holds (see examples below). Note that BSDEs with singular generator were already studied in \cite{jean:mast:poss:15} and \cite{jean:reve:14}, but the setting is completely different. 
\item Moreover if the condition \eqref{eq:alpha_gamma} holds for some $\ell > 1$, it remains true for any $1 \leq \ell'\leq \ell$. But with an additional cost in Condition \eqref{eq:f_growth_psi}.
\end{enumerate}
\end{rmrk}

\begin{xmpl}[Main example] \label{ex:control_ex}
{\rm 
The case given by equation \eqref{eq:example_control} has been developed in \cite{anki:jean:krus:13}, \cite{grae:hors:qiu:13} and \cite{krus:popi:15}. The process $a_t$ is $1/(q\alpha_t^q)$. All previous assumptions are satisfied if $\beta_t(e) \geq 0$ for any $t$ and $e$, $1/\alpha^{q}$ is in $L^1((0,T)\times \Omega)$, $\alpha \in L^\ell((0,T)\times \Omega)$ and if the non negative process $\gamma = f^0$ satisfies Condition \eqref{eq:alpha_gamma}. 

\hfill $\diamond$
}
\end{xmpl}

\begin{xmpl}[Toy example] \label{ex:toy_ex}
{\rm 
The function $f(y) = -y|y|^q$ satisfies all previous conditions. It corresponds to generator \eqref{eq:example_control} with $\alpha_t = (1/q)^{1/q}$, $\beta_t(e) = +\infty$ and $\gamma_t = 0$. 

\hfill $\diamond$
}
\end{xmpl}

\begin{xmpl}[With power singularity] \label{ex:power_sing}
{\rm
We will also discuss the case:
$$f(t,y) = -(T-t)^\varsigma y|y|^q + \frac{1}{(T-t)^\varpi}$$
where $\varsigma$ and $\varpi$ are two real numbers. Since $a_t = (T-t)^\varsigma$ is in $L^1(0,T)$, $\varsigma$ must be greater than $-1$. Note that this lower bound is necessary to have an optimal control in \eqref{eq:control_pb_intro} (see Example 1.1 in \cite{anki:jean:krus:13}). Condition \eqref{eq:alpha_gamma} imposes that 
$$\int_0^T (T-t)^{-\ell \varsigma/q} + (T-t)^{\ell (p-\varpi)} dt < +\infty.$$
This implies the following bounds: 
$$-1 < \varsigma < q, \qquad \varpi < 1 + 1/q + 1/\ell.$$
with $1 \leq \ell$ and $\ell < q/\varsigma$ if $\varsigma> 0$. 
The singularity ot time $T$ of the generator has to be not too important (upper bound on $\varpi$) and the coefficient $a$ before $y|y|^q$ can degenerate at time $T$, but not too quickly (upper bound on $\varsigma$). 

\hfill $\diamond$
}
\end{xmpl}

\subsection{Known results}
%---------------

In \cite{krus:popi:14}, we proved that if $\xi \in L^p(\Omega)$, for some $p \geq 2$, then under Conditions {\rm {\bf (A)}} there exists a unique solution $(Y,Z,U,M)$ in $\bS^p(0,T)$ to the BSDE \eqref{eq:BSDE_jumps}. In \cite{krus:popi:15}, the following result is proved. 
\begin{thrm}[Theorem 1 in \cite{krus:popi:15}]\label{thm:exists_super_sol} 
Under Conditions {\rm {\bf (A)}} there exists a process $(Y,Z,U,M)$ such that 
\begin{itemize}
\item $(Y,Z,U,M)$ belongs to $\bS^\ell(0,t)$ for any $t < T$.  
\item A.s. for any $t\in [0,T]$, $Y_t \geq 0$.
\item For all $0\leq s \leq t < T$:
\begin{equation*}
Y_s=Y_t + \int_s^t f(t,Y_r,Z_r,U_r) dr-\int_s^t Z_r dW_r-\int_s^t\int_E U_s(e) \tilde{\mu}(ds,de) + M_T-M_t.
\end{equation*}
\item If the filtration $\bF$ is quasi left continuous, $(Y,Z,U,M)$ is a super-solution in the sense that a.s. \eqref{eq:term_cond_super_sol} holds:
\begin{equation*}
\liminf_{t\to T} Y_t \geq \xi.
\end{equation*}
\end{itemize}
Any process $(\tilde Y, \tilde Z, \tilde U,\tilde M)$ satisfying the previous four items is called \textbf{super-solution} of the BSDE \eqref{eq:BSDE_jumps} with singular terminal condition $\xi$.
\end{thrm}
A key point in the construction made in \cite{krus:popi:15} is the following a priori estimate:
\begin{equation}\label{eq:a_priori_estimate}
Y_t \leq \frac{K_{\ell,L,\vartheta}}{(T-t)^{1+1/q}} \left\{\E \left( \ \int_t^{T} \left[ \left(\frac{1}{qa_s}\right)^{1/q} + (T-s)^{1+1/q} f^0_s \right]^{\ell} ds \bigg| \tri_t\right) \right\}^{1/\ell}
\end{equation}
where $K_{\ell,L,\vartheta}$ is a non negative constant depending only on $\ell$, $L$ and $\vartheta$ and this constant is a non decreasing function of $L$ and $\vartheta$ and a non increasing function of $\ell$. Condition \eqref{eq:alpha_gamma} implies that a.s. $Y_t < +\infty$ on $[0,T)$. 

\begin{rmrk}
$\ $
\begin{itemize}
\item The constants $K_{\ell,L,\vartheta}$ and $\ell > 1$ come from the growth condition on $f$ w.r.t. $z$ and $u$. 
\item If we assume that $f(t,0,z,u)$ is uniformly bounded from above by $K_f$, in \eqref{eq:a_priori_estimate} we can take $\ell=1$ and $K_{\ell,L,\vartheta}=1$ and we add $\frac{K_f}{2+1/q}(T-t)$ . 
\end{itemize}
\end{rmrk}

\begin{enonce}{Back to the examples.}
\begin{itemize}
\item Example \ref{ex:toy_ex}. If $f(y)=-y|y|^q$, $a_t =1$, $K_f=0$, and we obtain as in \cite{popi:06}:
$$Y_t \leq \left( \frac{1}{q(T-t)} \right)^{1/q}.$$
\item Example \ref{ex:power_sing}. Recall that $-1 < \varsigma < q$ and $\varpi < 2 + 1/q$. Here again one can take $K_f= 0$ and 
$$Y_t \leq \left(\frac{1}{q}\right)^{1/q}\frac{q}{q-\varsigma} \frac{1}{(T-t)^{(1+\varsigma)/q}} + \frac{1}{2+1/q- \varpi} \frac{1}{(T-t)^{\varpi -1}}.$$
\end{itemize}
\end{enonce}

Let us finish this section by the minimality of the solution.
\begin{prpstn}[Minimal solution]
The solution $(Y,Z,U,M)$ obtained by approximation is minimal. If $(\widetilde Y, \widetilde Z, \widetilde U, \widetilde M)$ is another non negative super-solution, then for all $t \in [0,T]$, $\Prb$-a.s. $\widetilde Y_t \geq Y_t$. 
\end{prpstn}

\vspace{0.5cm}
Now we give the main ideas of the proof of the existence result (Theorem \ref{thm:exists_super_sol}). It is important to study the behaviour of $Y$ in the next sections. The approach in \cite{krus:popi:15} is to approximate our BSDE by considering a terminal condition of the form $\xi^n:=\xi\wedge n$ and observe asymptotic behaviour.

In the rest of the paper, $(Y^n,Z^n,U^n,M^n)$ will be the solution of the truncated BSDE:
\begin{eqnarray}\label{eq:trunc_BSDE}
Y^n_t & = & \xi\wedge n + \int_t^T f_n(s,Y^n_s,Z^n_s,U^n_s) ds -\int_t^T Z^n_s dW_s \\ \nonumber
& -& \int_t^T\int_E U^n_s(e)\tilde{\mu}(ds,de) - (M^n_T-M^n_t).
\end{eqnarray}
Here $f_n(t,y,z,u)$ is the generator obtained by the truncation on $f^0_t$:
\begin{equation} \label{eq:truncated_gene}
f_n(t,y,z,u) = (f(t,y,z,u)-f^0_t) + (f^0_t \wedge n).
\end{equation}
Existence and uniqueness of $(Y^n,Z^n,U^n,M^n)$ comes from Theorem 2 in \cite{krus:popi:14}. Moreover using comparison argument (see \cite{krus:popi:14} or \cite{quen:sule:13}) we can obtain for $m \leq n$: $0\leq Y^m_t\leq Y^{n}_t $. And for any $n$, $Y^n$ satisfies Estimate \eqref{eq:a_priori_estimate} (Proposition 2 in \cite{krus:popi:15}). 
This allows us to define $Y$ as the increasing limit of the sequence
$(Y^n_t)_{n\geq 1}$: $$\forall\, t\in[0,T],\quad Y_t:=\lim_{n\rightarrow \infty} Y^n_t.$$
Proposition 3 in \cite{krus:popi:15} shows that there exists a constant $C$ such that for any $0 < t < T$
\begin{eqnarray}\label{eq:approx_estim_L2_main_ineq}
&&\E\left[\sup_{0\leq s\leq t} |Y^n_s-Y^m_s|^\ell+\left( \int_0^t |Z^n_s-Z^m_s|^2 ds \right)^{\ell/2} \right] \\ \nonumber
&& \qquad + \E \left[ \left( \int_0^t\int_E |U^n_s(e)-U^m_s(e)|^2 \mu(ds,de) \right)^{\ell/2} + [M^n_t - M^m_t]^{\ell/2}\right]\\ \nonumber
&& \\ \nonumber
&& \qquad \qquad \leq C\E\left[ |Y^n_t-Y^m_t|^\ell\right] + C \E \int_0^t |f^0_s\wedge n - f^0_s\wedge m|^\ell ds. 
\end{eqnarray}
Since $Y^n_t$ converges to $Y_t$ almost surely, with the a priori estimate \eqref{eq:a_priori_estimate}, Condition \eqref{eq:alpha_gamma} and Inequality \eqref{eq:approx_estim_L2_main_ineq}, thanks to the dominated convergence theorem, we can deduce:
\begin{enumerate}
\item For every $\varepsilon>0$, $(Y^n)_{n\geq1}$ converges to $Y$ in $\bD^\ell(0,T-\varepsilon)$.
\item  $(Z^n)_{n\geq1}$ is a Cauchy sequence in $\bH^\ell(0,T-\varepsilon)$, and converges
to $Z\in\bH^\ell(0,T-\varepsilon)$.
\item $(U^n)_{n\geq1}$ is a Cauchy sequence in $\bL^\ell_{\mu}(0,T-\varepsilon)$, and converges to $U\in\bL^\ell_{\mu}(0,T-\varepsilon)$. 
\item $(M^n)_{n\geq1}$ is a Cauchy sequence in $\bM^\ell(0,T-\varepsilon)$, and converges to $M\in\bM^\ell_{\mu}(0,T-\varepsilon)$. 
\end{enumerate}
The limit $(Y,Z,U,M)$ satisfies for every $0\leq t<T$, for all $0\leq s \leq t$
\begin{equation}\label{eq:bsde_dynamic}
Y_s=Y_t+\int_s^t f(r,Y_r,Z_r,U_r)dr -\int_s^t Z_r dB_r -\int_s^t \int_E U_r(e) \tilde{\mu}(dr,de) - M_t + M_s.
\end{equation}
and $Y$ satisfies Inequality \eqref{eq:a_priori_estimate}. Note that all these results are obtained without the quasi left-continuity assumption on the filtration $\bF$. But since the solution $(Y,Z,U,M)$ satisfies the dynamic \eqref{eq:bsde_dynamic} only on $[0,T-\eps]$ for any $\eps > 0$, we cannot derive directly the existence of a left limit at time $T$ for $Y$. 

\subsection{Condition on the filtration $\bF$} \label{ssect:filtration}
%-----------------

In the $L^p$-theory, we assume that the terminal condition $\xi$ belongs to $L^p(\Omega)$ and that the generator $f$ satisfies $f^0 \in L^p((0,T)\times \Omega)$ for some $p > 1$ (see \cite{bouc:poss:zhou:15} or \cite{krus:popi:14}). Then if $(Y,Z,U,M)$ is the solution of the BSDE \eqref{eq:BSDE_jumps} in $\bS^p$ (called $L^p$-solution), we have 
$$Y_{T^-} = \lim_{t\to T, \ t < T} Y_t = \xi -  (M_T - M_{T^-}) = \xi - \Delta M_T = Y_T - \Delta M_T.$$
In other words the orthogonal martingale $M$ could have a jump at time $T$ and thus \eqref{eq:equality_liminf_T} may not hold. For example, consider the filtration $\bF$ defined by
$$\tri_t = \{\emptyset,\Omega\}, \ t < 1, \qquad \tri_t = \sigma(X), \ t \geq 1, \ \mbox{with } \Prb(X=1)=\Prb(X=-1)=1/2.$$
$T$ is called a thin time\footnote{The author thanks Monique Jeanblanc for her precious information on this topic.} (see \cite{aksa:chou:jean:16} for the definition). Then for $T=1$, $f = 0$ and $\xi = \ind_{X=1}$, $Y_t= 1/2$ for any $t< T$. 

To ensure that \eqref{eq:equality_liminf_T} will be true, we need to impose an extra condition on the filtration $\bF$ to ensure that a martingale cannot have a jump at time $T$. A usual and enough condition is: the filtration $\bF$ is quasi left-continuous. For example if $\bF$ is generated by the Brownian motion and the Poisson random measure, this hypothesis is true. A sufficient and less strong condition is: the filtration $\bF$ is left-continuous at time $T$ (see the proof of Proposition 25.19 in \cite{kall:02}). The reader could find examples on non quasi left-continuous filtrations in Remark 1.9 \cite{aksa:chou:jean:16} (see also the references therein, in particular \cite{jaco:skor:94}). 

If we assume that such kind of assumption holds for $\bF$ and thus if a martingale cannot have a jump at time $T$, then we have a.s. 
$$ \xi \wedge n =\liminf_{t\to T} Y^n_t \leq \liminf_{t\to T} Y_t,$$
and thus immediately the singular minimal solution satisfies \eqref{eq:term_cond_super_sol}: a.s.
$$\liminf_{t\to T} Y_t \geq \xi.$$
To summarize, we add a condition on $\bF$ to be sure that the previous inequality holds. This inequality will be used in Section \ref{sect:continuity_Y} to prove continuity of $Y$.

\section{Existence of a left-limit} \label{sect:exist_limit}
%---------------

In this section Conditions \textbf{(A)} hold and we will prove that the left limit of $Y$ at time $T$ exists provided we know the precise behaviour of the generator w.r.t. $y$. In other words we show that $Y$ is c\`adl\`ag on $[0,T]$. In some sense our generator has to be more specific to control the behaviour of the supersolution at time $T$. 

\noindent \textbf{Assumption (B).} The generator satisfies
\begin{equation}\label{eq:spec_gene} \tag{B}
b_t g(y) \leq f(t,y,z,u)-f(t,0,z,u) , \quad \forall y \geq 0, \ \forall (t,z,u),
\end{equation}
where 
\begin{itemize}
\item $b$ is positive and $b \in L^1((0,T)\times \Omega)$;
\item $g$ is a negative, decreasing and of class $C^1$ function and concave on $\R_+$ with $g(0)<0$ and $g'(0)<0$.  
\end{itemize}
Since Conditions \textbf{(A)} should hold, in particular \eqref{eq:f_upper_bound}, from \textbf{(B)} we deduce that $b_t g(y) \leq -a_ty|y|^q$ for any $t \in [0,T]$ and $y$. Thus w.l.o.g. $g(y) \leq - y|y|^q$ and $ b_t \geq (-1/g(1))a_t = C a_t$ for some positive constant $C$. We can always add to $g$ a linear function like $- y-1$ such that $g(0)< 0$ and $g'(0) < 0$. 

In the sequel of this section, we decompose $f$ as follows:
$$ f(t,y,z,u) = \phi(t,y,z,u) + \pi(t,z,u) + f^0_t $$
where $f^0_t = f(t,0,0,0)$ and 
\begin{eqnarray*}
\phi(t,y,z,u) & = & f(t,y,z,u)-f(t,0,z,u)\\
\pi(t,z,u) & = & f(t,0,z,u)-f(t,0,0,0).
\end{eqnarray*}
\begin{thrm} \label{thm:exists_limit}
Assumptions {\rm \textbf{(A)}} and {\rm \textbf{(B)}} hold. Moreover one of the next three cases holds:
\begin{itemize}
\item {\bf Case 1.} $f$ does not depend on $u$ or $\pi(t,0,u) \geq 0$;
\item {\bf Case 2.} $\vartheta \in \bL^1_\lambda(E)$ and there exists a constant $\kappa_*>-1$ such that $\kappa^{0,0,u,0}_s(e) \geq \kappa_*$ a.e. for any $(s,u,e)$;
\item {\bf Case 3.} $\lambda$ is a finite measure on $E$.
\end{itemize}
Then the minimal supersolution $Y$ has a left limit at time $T$.
\end{thrm}

\begin{rmrk}
$\ $
\begin{enumerate}
\item Again this result shows that the process $Y$ is c\`adl\`ag on $[0,T]$ when the filtration $\bF$ is complete and right-continuous. No additional assumption (left-continuity) on the filtration is needed here. 
\item If Inequality \eqref{eq:term_cond_super_sol} holds, then a.s. $\displaystyle \lim_{t\to T} Y_t \geq \xi.$
\item The second condition on $\kappa$ in Case 2 is quite classical. Indeed a stronger version is used to prove the comparison principle for BSDE with jumps in \cite{barl:buck:pard:97} or in \cite{roye:06}.
\end{enumerate}
\end{rmrk}

\begin{enonce}{Back to the examples.}
Clearly Conditions of Theorem \ref{thm:exists_limit} hold for the three examples \ref{ex:control_ex}, \ref{ex:toy_ex} and \ref{ex:power_sing}.  

Indeed for Example \ref{ex:control_ex}, $\lambda$ is supposed to be finite. Moreover since $\psi$ defined by \eqref{eq:generater_control_psi} is Lipschitz w.r.t. $y$ we obtain for $y\geq 0$:
\begin{eqnarray*}
f(t,y,z,u)-f(t,0,z,u) & = & - \frac{y|y|^{q}}{q \alpha_t^{q}} - \psi(t,y,u) +\psi(t,0,u) \geq  - \frac{y|y|^{q}}{q \alpha_t^{q}} - L |y| \\
& \geq & -\left( \frac{1}{q\alpha_t^{q}} \vee L \right) \left( y^{1+q} + y \right) \geq b_t g(y).
\end{eqnarray*}
if
$$b_t = \frac{1}{q\alpha_t^{q}} \vee L ,\qquad g(y) = -y^{1+q}  - y - 1.$$
For the examples \ref{ex:toy_ex} and \ref{ex:power_sing}, $g(y) = -y|y|^q-y-1$ also works. Since $f$ does not depend on $u$, there is no restriction on $\lambda$ or on $\kappa$. 
\end{enonce}

Let us just give the trick of the proof of the previous theorem. If $b_t$ is deterministic, consider the ordinary differential equation $y' = -f(t,y)=-b_t g(y)$. To solve it, we can separate the variables and with $\Theta' = 1/g$, we write formally:
$$\Theta(y(T)) - \Theta(y(t)) = - \int_t^T \frac{y'(s)}{g(y(s))} ds =  -\int_t^T b_s ds$$
which gives:
$$y(t) = \Theta^{-1} \left( \Theta(y(T)) + \int_t^T b_s ds\right).$$
We will follow the same idea: we apply the It\^o formula with the function $\Theta$ to the process $Y_t$. Then we cancel the martingale part with the conditional expectation and we have to control the terms of finite variations. The positive parts will give a non negative supermartingale, which has always a limit at time $T$. The negative parts have to be more carefully studied to prove that they have a limit at time $T$. This is the reason why we impose these extra conditions on $f$, $\kappa$ or $\lambda$.

Now let us go into details and define the function $\Theta$ on $(0,+\infty)$ by
\begin{equation}
\Theta (x) = \int_{x}^{+\infty} \frac{-1}{g(y)} dy .
\end{equation}
Recall that $g$ is continuous and negative on $\R_+$. Thus from the condition $g(y) \leq -y|y|^q$, the function $\Theta : [0,+\infty) \to (0,\Theta(0)]$ is well defined, non increasing, of class $C^{1}$, and bijective. Let $\Xi : (0,\Theta(0)] \to [0,+\infty)$ be the inverse of $\Theta$. Let us give some explicit examples.
\begin{itemize}
\item If $g(y) = - y^2 - 2y -1$, $\Theta (x) = (x+1)^{-1}$ and $\Xi(x) = (1/x)-1$. 
\item If $g(y)=  - \exp (y)$, for $y \geq 0$, $\Theta(x)  = \exp (-x)$ and $\Xi(x) = -\ln(x)$. 
\end{itemize}
We proceed as in \cite{popi:06} (see here for more details) and we apply the function $\Theta$ to the process $Y^n$, where $(Y^n,Z^n,U^n,M^n)$ is the solution of the truncated BSDE \eqref{eq:trunc_BSDE}. 
\begin{lmm} \label{prop:exists_limit}
Assume that the conditions of Theorem \ref{thm:exists_limit} are satisfied. Then the process $Y$ can be written as follows:
$$Y_t=\Theta^{-1} \left(\E^{\tri_t}\left[\Theta(\xi) \right]+ \psi_t^- -\psi^+_t\right)$$
where $\psi^+$ and $\psi^-$ are two non negative c\`adl\`ag  supermartingales with a.s. $\displaystyle \lim_{t\to T} \psi^-_t = 0$. 
\end{lmm}
\begin{proof}
Since $Y^n_t$ is bounded from below by zero, we can apply It\^o's formula:
\begin{eqnarray} \label{eq:Ito_form_Theta} 
&& \Theta(Y^n_t) = \Theta(\xi\wedge n) + \int_t^T \Theta'(Y^n_{s^-}) f_n(s,Y^n_s,Z^n_s,U^n_s) ds \\ \nonumber
&&\quad  -  \int_t^T \Theta'(Y^n_{s^-}) Z^n_s dW_s - \int_t^T  \Theta'(Y^n_{s^-})  \int_E U^n_s(e) \tilde{\mu}(de,ds)- \int_t^T  \Theta'(Y^n_{s^-})  dM^n_s \\ \nonumber
&&\quad  - \frac{1}{2} \int_t^T \Theta''(Y^n_{s^-}) |Z^n_s|^2 ds - \frac{1}{2} \int_t^T \Theta''(Y^n_{s^-}) d[M^n]^c_s \\ \nonumber
&&\quad  -  \int_t^T\int_E\left[\Theta(Y^n_{s^-} + U^n_s(e))-\Theta(Y^n_{s_-})-\Theta'(Y^n_{s_-})U^n_s(e)\right]\mu(ds,de) \\ \nonumber
&&\quad  -  \sum_{t<s \leq T} \left[\Theta(Y^n_{s^-}+ \Delta M^n_s)-\Theta(Y^n_{s_-})-\Theta'(Y^n_{s_-})\Delta M^n_s \right] \\ \nonumber
&& = \E^{\tri_t} \Theta(\xi\wedge n) - \psi^n_t
\end{eqnarray}
where 
\begin{eqnarray*}
\psi^n_t & = &-\E^{\tri_t}  \int_t^T \Theta'(Y^n_{s}) f_n(s,Y^n_s,Z^n_s,U^n_s) ds +\frac{1}{2} \E^{\tri_t} \int_t^T \Theta''(Y^n_s) |Z^n_s|^2 ds \\
& + &\frac{1}{2}\E^{\tri_t}  \int_t^T \Theta''(Y^n_s) d[M^n]^c_s  + \E^{\tri_t} \sum_{t<s \leq T} \left[\Theta(Y^n_{s^-}+ \Delta M^n_s)-\Theta(Y^n_{s_-})-\Theta'(Y^n_{s_-})\Delta M^n_s \right]\\
&+ & \E^{\tri_t}\int_t^T \int_E\left[\Theta(Y^n_{s^-} + U^n_s(e))-\Theta(Y^n_{s_-})-\Theta'(Y^n_{s_-})U^n_s(e)\right]\mu(ds,de) .
\end{eqnarray*}
Since $\Theta$ is non increasing, we estimate now the following difference for $m\geq n$:
\begin{eqnarray*}
0 \leq \Theta(Y^n_t)-\Theta(Y^m_t) & =& \E^{\tri_t}\left[\Theta(\xi\wedge n) -\Theta(\xi\wedge m)\right] -(\psi^n_t - \psi^m_t).
\end{eqnarray*}
And so we obtain:
$$|\psi^n_t-\psi^m_t|\leq \E^{\tri_t}\left[\Theta(\xi\wedge n) -\Theta(\xi\wedge m) \right]\vee \left[ \Theta(Y^n_t)-\Theta(Y^m_t)\right].$$
Since the sequences $(\E^{\tri_t}\left[\Theta(\xi\wedge n)\right])_{n\geq1}$ and $(\Theta(Y^n_t))_{n\geq 1}$ converge a.s. and in $\mathbb{L}^1$ (by monotone convergence theorem), we deduce that $(\psi^n_t)_{n\geq1}$ converge a.s. and in $\mathbb{L}^1$ to some $\psi_t$. So by passing to the limit, one can write:
\begin{equation} \label{eq:explicit_expr_Y}
\Theta(Y_t)=\E^{\tri_t}[\Theta(\xi)] -\psi_t .
\end{equation}
Our aim now is to prove that the negative part of $\psi^n_t$ is bounded with an upper bound independent of $n$. 

Let us recall the decomposition of the generator $f_n$:
\begin{eqnarray*}
f_n(s,y,z,u) & = & \left[f(s,y,z,u)-f(s,0,z,u)\right] + \left[f(s,0,z,u)-f(s,0,0,0)\right] +  (f^0_s \wedge n)\\
& =& \phi(s,y,z,u) + \pi(s,z,u) +  (f^0_s \wedge n).
\end{eqnarray*}
Recall that 
\begin{eqnarray}\label{eq:martingale_to_control_1} 
\psi^n_t & = &-\E^{\tri_t}  \int_t^T \Theta'(Y^n_{s}) \phi(s,Y^n_s,Z^n_s,U^n_s) ds \\ \nonumber
&+ & \E^{\tri_t}\int_t^T \int_E\left[\Theta(Y^n_{s^-} + U^n_s(e))-\Theta(Y^n_{s_-})-\Theta'(Y^n_{s_-})U^n_s(e)\right]\mu(ds,de) \\ \nonumber
&-& \E^{\tri_t}  \int_t^T \Theta'(Y^n_{s^-}) \pi(s,Z^n_s,U^n_s) ds +  \frac{1}{2} \E^{\tri_t} \int_t^T \Theta''(Y^n_s) |Z^n_s|^2 ds \\ \nonumber
&-& \E^{\tri_t}  \int_t^T \Theta'(Y^n_s) (f^0_s \wedge n) ds  +\frac{1}{2}\E^{\tri_t}  \int_t^T \Theta''(Y^n_s) d[M^n]^c_s  \\ \nonumber
& + &  \E^{\tri_t} \sum_{t<s \leq T} \left[\Theta(Y^n_{s^-}+ \Delta M^n_s)-\Theta(Y^n_{s_-})-\Theta'(Y^n_{s_-})\Delta M^n_s \right].
\end{eqnarray}
Recall that $\Theta' = 1/g < 0$ and $\Theta'' = -g'/g^2 \geq 0$. Hence $\Theta$ is non decreasing and convex and
$$\frac{1}{2}\E^{\tri_t}  \int_t^T \Theta''(Y^n_s) d[M^n]^c_s  + \E^{\tri_t} \sum_{t<s \leq T} \left[\Theta(Y^n_{s^-}+ \Delta M^n_s)-\Theta(Y^n_{s_-})-\Theta'(Y^n_{s_-})\Delta M^n_s \right]\geq 0$$
and
$$- \E^{\tri_t}  \int_t^T \Theta'(Y^n_s) (f^0_s \wedge n) ds \geq 0.$$
The last three terms in \eqref{eq:martingale_to_control_1} are non negative. 

Then we control the first term using Condition \textbf{(B)}:
\begin{equation}\label{eq:estim_neg_part_0}
-\Theta'(Y^n_{s}) \phi(s,Y^n_s,Z^n_s,U^n_s) \geq - \Theta'(Y^n_{s}) b_s g(Y^n_s) = -b_s.
\end{equation}
For the remaining terms in \eqref{eq:martingale_to_control_1} we have
\begin{eqnarray*}
\Theta'(y) \pi(t,z,u)& =& \Theta'(y)(\pi(t,z,u)-\pi(t,0,u))+ \Theta'(y)\pi(t,0,u).
\end{eqnarray*}
Hence the terms containing $Z^n$ are:
\begin{eqnarray*}
&& -\Theta'(Y^n_{s^-}) (\pi(s,Z^n_s,U^n_s)-\pi(s,0,U^n_s)) + \frac{1}{2} \Theta''(Y^n_s)|Z^n_s|^2 \\
&&\quad \geq -L (-\Theta'(Y^n_s)) |Z^n_s| + \frac{1}{2} \Theta''(Y^n_s)|Z^n_s|^2= \frac{-g'(Y^n_s)}{2} \frac{|Z^n_s|^2}{|g(Y^n_s)|^2} + L \frac{|Z^n_s|}{g(Y^n_s)} \\
&& \quad = \frac{-g'(Y^n_s)}{2} \left( \frac{|Z^n_s|}{g(Y^n_s)} -\frac{L}{g'(Y^n_s)} \right)^2  + \frac{L^2}{2g'(Y^n_s)} \geq \frac{L^2}{2g'(Y^n_s)} .
\end{eqnarray*}
We have used  Condition \eqref{eq:f_lip_z} such that:
$$\pi(s,Z^n_s,U^n_s)-\pi(s,0,U^n_s) \geq -L |Z^n_s|.$$
Since $g$ is concave, $g'$ is non increasing. Thus we have $g'(Y^n_s) \leq g'(0)$ and
\begin{equation} \label{eq:estim_neg_part}
-\Theta'(Y^n_{s^-}) (\pi(s,Z^n_s,U^n_s)-\pi(s,0,U^n_s)) + \frac{1}{2} \Theta''(Y^n_s)|Z^n_s|^2 \geq \frac{L^2}{2g'(0)}.
\end{equation}
Thereby in \eqref{eq:martingale_to_control_1} the last term to control is:
\begin{eqnarray*} \nonumber
&& \E^{\tri_t}\int_t^T \int_E\left[\Theta(Y^n_{s^-} + U^n_s(e))-\Theta(Y^n_{s_-})-\Theta'(Y^n_{s_-})U^n_s(e)\right]\mu(ds,de) \\
&&\qquad  - \E^{\tri_t}  \int_t^T \Theta'(Y^n_{s^-}) \pi(s,0,U^n_s) ds .
\end{eqnarray*}

\begin{itemize}
\item \noindent \textbf{Case 1:} Assume that $f$ does not depend on $u$ or $\pi(s,0,u)\geq 0$ for any $s$ and $u$. From the convexity of $\Theta$, the integral w.r.t. $\mu$ is non negative. Hence using \eqref{eq:estim_neg_part_0} and \eqref{eq:estim_neg_part}, the negative part of $\psi^n$ is controlled for any $n$ by:
\end{itemize}
\begin{equation} \label{eq:control_neg_part_1}
(\psi^{n}_t)^- \leq -  \frac{L^2}{2g'(0)} (T-t) + \E^{\tri_t}\int_t^T b_s ds.
\end{equation}
Let us deal now with the cases 2 and 3. Up to some localization procedure we have
\begin{eqnarray*}
&& \E^{\tri_t}\int_t^T \int_E\left[\Theta(Y^n_{s^-} + U^n_s(e))-\Theta(Y^n_{s_-})-\Theta'(Y^n_{s_-})U^n_s(e)\right]\mu(ds,de) \\
&& \quad =  \E^{\tri_t}\int_t^T \int_E\left[\Theta(Y^n_{s^-} + U^n_s(e))-\Theta(Y^n_{s_-})-\Theta'(Y^n_{s_-})U^n_s(e)\right]\lambda(de) ds.
\end{eqnarray*}
With Assumption \eqref{eq:f_jump_comp} we obtain:
$$ -\Theta'(Y^n_{s^-}) \pi(s,0,U^n_s) \geq  - \Theta'(Y^n_{s^-}) \int_E \kappa^{0,0,U^n,0}_s U^n_s(e) \lambda(de).$$
For simplicity, $\kappa^{0,0,U^n,0}_s(e)$ will be denoted by $\kappa^n_s(e)$. 
The jump part is bounded from below by the following process
\begin{eqnarray*}
\E^{\tri_t}\int_t^T \int_E\left[\Theta(Y^n_{s^-} + U^n_s(e))-\Theta(Y^n_{s_-})-\Theta'(Y^n_{s_-})(1+\kappa^{n}_s(e))U^n_s(e)\right]\lambda(de) ds.
\end{eqnarray*}
From Condition \eqref{eq:f_jump_comp}, $1+\kappa^{n}_s(e) \geq 0$. Thus for a fixed $y>0$, if $\kappa^{n}_s(e) > -1$, the function
$$u \mapsto \Theta(y+u)-\Theta(y)-\Theta'(y)(1+\kappa^{n}_s(e))u$$
has a minimum $m$ on $(-y,+\infty)$ at the point $u^*$ satisfying 
$$\Theta'(y+u^*)=\Theta'(y)(1+\kappa^{n}_s(e)) \Longleftrightarrow g(y+u^*) = \frac{g(y)}{1+\kappa^{n}_s(e)}.$$
In other words:
$$u^* = -y + g^{-1} \left( \frac{g(y)}{1+\kappa^{n}_s(e)}\right) = g^{-1}  \left( \frac{g(y)}{1+\kappa^{n}_s(e)}\right) - y.$$
This minimum is equal to
$$m=- (u^*)^2 \int_0^1 \rho \Theta''(y+ \rho u^*) d\rho \leq 0.$$
If $\kappa^n_s (e) = -1$, then $u^*=+\infty$ and $m=-\Theta(y)$. In any case $\kappa^n_s (e) u^*  \leq 0$.  

\vspace{0.3cm}
\begin{itemize}
\item \noindent \textbf{Case 2:} $\vartheta \in \bL^1_\lambda(E)$ and $\kappa^{y,z,u,v}_t(e)$ is bounded from below by some constant $\kappa_* >-1$ uniformly w.r.t. all parameters. 

By convexity of $\Theta$, we obtain
\begin{eqnarray*}
&& \Theta(y+u^*)-\Theta(y)-\Theta'(y)(1+\kappa^{n}_s(e))u^* \geq -\Theta'(y)\kappa^{n}_s(e)u^*.
\end{eqnarray*}
If $\kappa^{n}_s(e)\geq 0$, 
$$ -\Theta'(y)\kappa^{n}_s(e)u^* \geq y\Theta'(y) \kappa^{n}_s(e) \geq y\Theta'(y) \vartheta(e)$$
and if $\kappa_* \leq \kappa^{n}_s(e)\leq 0$
\begin{eqnarray*}
-\Theta'(y)\kappa^{n}_s(e)u^* & \geq & -\Theta'(y) \kappa^{n}_s(e)g^{-1} \left( \frac{g(y)}{1+\kappa^{n}_s(e)}\right) =\Theta'(y)|\kappa^{n}_s(e)| g^{-1} \left( \frac{g(y)}{1+\kappa^{n}_s(e)}\right) \\
& \geq & \Theta'(y)|\kappa^{n}_s(e)| g^{-1} \left( \frac{g(y)}{1+\kappa_*}\right) \geq \frac{1}{g(y)} g^{-1} \left( \frac{g(y)}{1+\kappa_*}\right) \vartheta(e) .
\end{eqnarray*}
From our assumption on $g$, the functions $y\Theta'(y)=y/g(y)$ and $\frac{1}{g(y)} g^{-1} \left( \frac{g(y)}{1+\kappa_*}\right) $ are non positive and bounded from above respectively by a constant $-K_g<0$ depending on $g$ and by $-K_{g,\kappa_*}< 0$ depending only on $g$ and $\kappa_*$. This last estimate and the inequalities \eqref{eq:estim_neg_part_0} and \eqref{eq:estim_neg_part} imply that 
\end{itemize}

\begin{equation} \label{eq:control_neg_part_2}
(\psi^{n}_t)^- \leq \left[- \frac{L^2}{2g'(0)}  + \|\vartheta\|_{L^1_\lambda} (K_g \vee K_{g,\kappa_*}) \right] (T-t) + \E^{\tri_t}\int_t^T b_s ds.
\end{equation}

\vspace{0.3cm}
\begin{itemize}
\item \noindent \textbf{Case 3:} $\lambda$ is a finite measure.

Since $u^* \geq -y$, then 
\begin{eqnarray*}
&& \Theta(y+u^*)-\Theta(y)-\Theta'(y)(1+\kappa^{n}_s(e))u^* \geq -\Theta(y)+\Theta'(y)(1+\kappa^{n}_s(e))y\\
&&\qquad  \geq -\Theta(y) + \frac{y}{g(y)} (1+\vartheta(e)).
\end{eqnarray*}
Since $-\Theta$ is non decreasing and since $Y^n$ is bounded from below by $0$, the inequalities  \eqref{eq:estim_neg_part_0}, \eqref{eq:estim_neg_part} and the assumption \eqref{eq:f_growth_psi} imply that 
\end{itemize}
\begin{eqnarray}\label{eq:control_neg_part_3}
(\psi^{n}_t)^- &\leq& \E^{\tri_t}\int_t^T b_s ds + (T-t) \left[\frac{-L^2}{2g'(0)} + \lambda(E) ( \Theta(0) + K_g) +\|\vartheta\|_{L^1_\lambda} K_g\right].
\end{eqnarray}
In the three cases \eqref{eq:control_neg_part_1}, \eqref{eq:control_neg_part_2} or \eqref{eq:control_neg_part_3}, the negative part of $\psi^n$ is bounded uniformly w.r.t. $n$ and since $b\in L^1((0,T)\times \Omega)$ the right-hand side of the three estimates goes to zero as $t$ tends to $T$.

Let us now conclude. Recall that $\psi^n$ converges to $\psi$ and \eqref{eq:explicit_expr_Y} holds. The estimates \eqref{eq:control_neg_part_1}, \eqref{eq:control_neg_part_2} or \eqref{eq:control_neg_part_3} show that the negative part of $\psi^n$ converges a.s. and in $L^1$ to the non negative c\`adl\`ag bounded supermartingale $\psi^-$. Moreover the limit of $\psi^-_t$ at time $T$ is equal to zero. The equation \eqref{eq:explicit_expr_Y} shows that the positive part of $\psi$ is c\`adl\`ag and a supermartingale by convergence of $(\psi^n)^+$. This achieves the proof of the proposition.
\end{proof}

Theorem \ref{thm:exists_limit} can be now proved immediately. $\psi^+$ being a non negative c\`adl\`ag supermartingale, we can deduce the existence of the following limit:
$$\psi_{T_-}^+:=\lim_{t\nearrow T}\psi^+_t$$
And so $Y_{T_-}$ exists and is equal to:
$$Y_{T_-}:=\lim_{t\nearrow T} Y_t= \Theta^{-1}\left( \Theta(\xi) -\psi_{T_-}^+ \right).$$
Let us remark that in the extreme case where $\kappa_s(e)=-1$, then $m = -\Theta(y)$ for $u=+\infty$ and thus the condition $\lambda(E) < +\infty$ is an almost necessary condition to ensure that the negative part of $\psi^n$ is finite.

\section{Continuity at time $T$} \label{sect:continuity_Y}
%-----------------

The second important result is the proof of Equality \eqref{eq:equality_liminf_T}. Here we deal with a general filtration $\bF$ and two singularities: one due to $\xi$, another due to the generator $f$. As mentioned in Section \ref{ssect:filtration}, %{sect:setting}, 
the filtration satisfies a condition to ensure that \eqref{eq:term_cond_super_sol} holds a.s.
$$\liminf_{t\to T} Y_t \geq \xi.$$

\subsection{Singularity of the generator}
%-----------------

Recall that the generator $f$ of the BSDE \eqref{eq:BSDE_jumps} can be {\it singular} in the sense that Condition \eqref{eq:alpha_gamma} implies 
$$\E \int_0^T (T-s)^{\ell p} (f^0_s)^\ell ds < +\infty.$$
Thus $f^0 \in L^1((0,T-\eps)\times \Omega)$ for any $\eps > 0$, but we could have $f^0_T=+\infty$ and/or $f^0 \not\in L^1((0,T)\times \Omega)$. In Example \ref{ex:power_sing} $f^0_t = (T-t)^{-\varpi}$
with $1 \leq \ell$ and $\varpi < 1 + 1/q + 1/\ell$. Hence for $\varpi \geq 1$, then $f^0 \not\in L^1((0,T)\times \Omega)$. The next result shows that Equality \eqref{eq:equality_liminf_T}
$$\liminf_{t\to T} Y_t = \xi$$
may be false.
\begin{lmm} \label{lem:counter_example}
Assume that the generator is given by: $f(t,y,z,u)=f(t,y) = -y|y|^q + f^0_t$ with $f^0$ deterministic and not in $L^1(0,1)$. Then a.s. $\displaystyle \lim_{t\to T} Y_t = +\infty$.
\end{lmm}
\begin{proof}
Recall that $(Y^n,Z^n,U^n,M^n)$ is solution of BSDE \eqref{eq:trunc_BSDE}
\begin{eqnarray*}
Y^n_t & = & \xi\wedge n - \int_t^T Y^n_s|Y^n_s|^q ds  + \int_t^T \left(f^0_s \wedge n \right) ds \\
& - & \int_t^T Z^n_s dW_s - \int_t^T\int_E U^n_s(e)\tilde{\mu}(ds,de) - (M^n_T-M^n_t).
\end{eqnarray*}
We define $R^n_t = \exp\left(-\int_0^t |Y^n_r|^q dr \right)$ and by It\^o formula:
$$R^n_t Y^n_t = \E \left[ R^n_T (\xi\wedge n) + \int_t^T R^n_s \left(f^0_s \wedge n \right) ds \bigg| \tri_t \right].$$
Hence we obtain
$$Y^n_t \geq \E \left[ \int_t^T \exp\left(-\int_t^s |Y^n_r|^q dr \right) \left(f^{0}_s \wedge n \right) ds \bigg| \tri_t \right] .$$
Since $Y^n\leq Y$ for any $n$, 
$$Y_t \geq Y^n_t \geq \E \left[ \int_t^T \exp\left(-\int_t^s |Y_r|^q dr \right) \left( f^0_s \wedge n \right) ds \bigg| \tri_t \right] .$$
Finally using Fatou lemma and since $f^0$ is deterministic, we have
$$Y_t \geq \int_t^T \E \left[  \exp\left(-\int_t^s |Y_r|^q dr \right) \bigg| \tri_t \right] f^0_s ds \geq \E \left[  \exp\left(-\int_t^T |Y_r|^q dr \right) \bigg| \tri_t \right]\int_t^T f^0_s ds .$$
From Theorem \ref{thm:exists_limit}, $Y$ is c\`adl\`ag on $[0,T]$. Hence $Y_t$ is finite a.s. if and only if $\lim_{t\to T} Y_t = +\infty$ a.s.
\end{proof}

Again for Example \ref{ex:power_sing} with $\varsigma=0$ and $\varpi\geq 1$, Equality \eqref{eq:equality_liminf_T} can not be true whatever the terminal condition $\xi$ is. Hence in the rest of this section, we will assume that 
\begin{equation}\label{eq:f0_int} \tag{A8}
f^0 \in L^1((0,T)\times \Omega).
\end{equation}

\subsection{Behaviour of $Y$}
%--------------

Again we now assume that under suitable assumptions on the filtration $\bF$, \eqref{eq:term_cond_super_sol} holds a.s. and we want to prove that the inequality is an equality. As explained in the introduction, in \cite{popi:06} we were able to prove this in the Brownian setting:
\begin{itemize}
\item When $q>2$ without additional conditions since we have a suitable control of $Z$.
\item When $q\leq 2$ but with Malliavin calculus: roughly speaking $Z$ is the Malliavin derivative of $Y$ and we use the integration by parts to remove $Z$ (Lemma 1.2.2 in \cite{nual:06}). 
\end{itemize}
In our general setting, we deal here only with the first case. Hence we need some estimate on $Z$ and $U$, which will be obtained if 
we strengthen Assumption \eqref{eq:alpha_gamma}: for some $\eta < 1$
\begin{equation}\label{eq:alpha_gamma_2}\tag{A6*}
\E \int_0^T (T-s)^{-1+\eta} \left[ \left(\frac{1}{qa_s}\right)^{1/q} + (T-s)^{1+1/q} f^0_s \right]^{\ell}ds < +\infty.
\end{equation}
If $f$ satisfies all conditions \eqref{eq:f_mono} to \eqref{eq:f_growth_psi}, with \eqref{eq:alpha_gamma_2} instead of \eqref{eq:alpha_gamma}, we say that $f$ satisfies \textbf{Conditions (A*)}. 
\begin{rmrk}[on Assumption \eqref{eq:alpha_gamma_2}]
$\ $ 
\begin{itemize}
\item Since $a\in L^1((0,T)\times \Omega)$, \eqref{eq:alpha_gamma_2} implies that $\eta + \ell/q > 0$.
\item If $f^0$ and $(1/a)^{1/q}$ are in $L^\ell ((0,T)\times \Omega)$, then \eqref{eq:alpha_gamma_2} holds for any $0 < \eta< 1$. The case $a$ bounded from below by a positive constant and $f^0$ bounded from above is a particular case (see Example \ref{ex:toy_ex}).
\item For Example \ref{ex:power_sing}, $a_t = (T-t)^\varsigma$ and $f^0_t = (T-t)^{-\varpi}$ with $-1 < \varsigma < q$, $1 < \ell < q/\varsigma$ and $\varpi <1+1/q+1/\ell$. Condition \eqref{eq:alpha_gamma_2} holds if we take $\eta$ such that 
$$\ell \max \left( \frac{\varsigma}{q}, -(1+1/q-\varpi) \right) < \eta < 1.$$
\end{itemize}
\end{rmrk}

The stronger condition \eqref{eq:alpha_gamma_2} implies the next result.
\begin{prpstn} \label{prop:sharp_estim_Z_U}
Under Conditions {\bf (A*)}, there exists a constant $C$ independent of $n$ such that the process $(Z^n,U^n)$ satisfies:
$$\E\left[ \int_0^T (T-s)^{\rho} \left( |Z^n_s|^2 + \|U^n_s\|^2_{\bL^2_\lambda} \right) ds  \right]^{\ell/2} \leq C.$$
The constant $\rho$ is given by:
\begin{equation}\label{eq:def_rho} 
\rho =\frac{2}{q} + 2\left( 1 -\frac{1}{\ell} \right) + \frac{2\eta}{\ell} .
\end{equation}
\end{prpstn}
The proof of this proposition is postponed to the next section. In the sequel we will need this sharper estimate on $Z$ and $U$ but with the technical condition 
\begin{equation}\label{eq:cond_rho} \tag{A9}
\rho =\frac{2}{q} + 2\left( 1 -\frac{1}{\ell} \right) + \frac{2\eta}{\ell} < 1
\end{equation}
This condition $\rho < 1$ is a balance between the non linearity $q$ and the singularity of the generator $f$. 
\begin{rmrk}[on Condition \eqref{eq:cond_rho}]
$\ $

\begin{enumerate}
\item If $f^0$ and $(1/a)^{1/q}$ are in $L^\ell ((0,T)\times \Omega)$, then \eqref{eq:alpha_gamma_2} holds for any $0 < \eta< 1$. Then $\rho<1$ for $\ell < 2$ and $q > \frac{2\ell}{2-\ell}.$ 
\item In particular if the generator is $f(y) = - y|y|^q$ (Example \ref{ex:toy_ex}), then $\rho < 1$ if $q > 2$, which was supposed in \cite{popi:06}.
\item In Example \ref{ex:power_sing}, the constant $\rho$ satisfies:
$$2 \max \left(\frac{(1+\varsigma)}{q}, -(1-\varpi) \right)+ 2\left( 1 -\frac{1}{\ell} \right)  < \rho.$$
The constant $\ell > 1$ can be chosen close to 1. Thus $\rho <1$ if 
$$2 \max \left( \frac{(1+\varsigma)}{q} , -(1-\varpi) \right) < 1.$$
Hence $\varpi < 3/2$ and $q > 2(1+\varsigma)$. In other words $f^0$ cannot be too singular at time $T$. Moreover the less degenerate is the process $a_t$, in other words the smaller is $\varsigma$, the smaller can be the non linearity coefficient $q$.
\end{enumerate}
\end{rmrk}

Now we work in the half-Markovian setting and we define the function $\Phi$ on $\R^d$ with values in $\R_+ \cup \{+\infty\}$ and with
$$\mathcal{S}=\{ x\in \mathbb{R}^d \quad s.t.\quad \Phi(x)=\infty\}$$
the set of singularity points for the terminal condition induced by $\Phi$. This set $\mathcal{S}$ is supposed to be closed. We also denoted by $\bord$ the boundary of $\mathcal{S}$. 

Our terminal condition $\xi$ satisfies \textbf{Conditions (C)} if 
\begin{equation}\label{eq:markov_setting} \tag{C1}
\xi=\Phi(X_T).
\end{equation} 
and if for all closed set $\mathcal{K} \subset \R^{d} \setminus \cS$ 
\begin{equation}  \label{eq:hyp_g}\tag{C2}
\Phi(X_{T}) \mathbf{1}_{\mathcal{K}}(X_{T}) \in \ L^{1} \left( \Omega, \tri_{T}, \Prb  \right).
\end{equation}
The process $X$ is the solution of a SDE with jumps: 
\begin{equation} \label{eq:SDE}
X_t=X_0+\int_0^t b(s,X_s)ds +\int_0^t \sigma(s,X_s)dW_s+\int_0^t\int_E h(s,X_{s_-},e)\tilde{\mu}(de,ds).
\end{equation}
The coefficients $b:\Omega \times [0,T] \times \R^d\to \R^d$, $\sigma:\Omega \times [0,T] \times \R^d\to \R^{d\times d}$ and $h:\Omega \times [0,T] \times \R^d\times E \to \R^d$ satisfy  \textbf{Assumptions (D)}:
\begin{enumerate}
\item $b$, $\sigma$ and $h$ are jointly continuous w.r.t. $(t,x)$ and Lipschitz continuous w.r.t. $x$ uniformly in $t$, $e$ or $\omega$, i.e. there exists a constant $K_{b,\sigma}$ or $K_h$ such that for any $(\omega,t,e) \in \Omega \times [0,T] \times E$, for any $x$ and $y$ in $\R^d$: a.s.
\begin{equation}\label{eq:lipsch_cond} \tag{D1}
|b(t,x)-b(t,y)| + |\sigma(t,x)-\sigma(t,y)|  \leq K_{b,\sigma} |x-y|
\end{equation}
and
\begin{equation}\label{eq:lipsch_cond_jump} \tag{D2}
|h(t,x,e)-h(t,y,e)|\leq K_h |x-y|(1\wedge |e|).
\end{equation}
\item $b$ and $\sigma$ growth at most linearly:
\begin{equation}\label{eq:growth_cond} \tag{D3}
|b(t,x)| + |\sigma(t,x)|  \leq C_{b,\sigma}(1+ |x|).
\end{equation}
\item $h$ is bounded w.r.t. $t$ and $x$ and there exists a constant $C_h$ such that a.s.
\begin{equation}\label{eq:growthcond_jumps} \tag{D4}
|h(t,x,e)| \leq C_h (1\wedge |e|).
\end{equation}
\end{enumerate}
Under Assumptions \textbf{(D)}, the forward SDE \eqref{eq:SDE} has a unique strong solution $X$ (see \cite{okse:sule:07} or \cite{prot:04}). To lighten the notation, the dimensions of $X$ and of the Brownian motion are the same. But this condition does not matter and we can also work with different dimensions.

In order to prove that $\displaystyle \liminf_{t\to T} Y_t = \xi$, we proceed as in \cite{popi:06}. But there is an extra term due to the covariance between the jumps of the SDE \eqref{eq:SDE} and the jumps of the BSDE \eqref{eq:BSDE_jumps}. To control this additional part, we make a link between the singularity set $\cS$ and the jumps of the forward process $X$. More precisely we assume 

\vspace{0.5cm}
\noindent \textbf{Conditions (E)}.
\begin{enumerate}
\item[(E1).] The boundary $\bord$ is compact and of class $C^2$. 
\item[(E2).] For any $x\in\mathcal{S}$, any $s \in [0,T]$ and $\lambda$-a.s. 
$$x+\beta(s,x,e)\in\mathcal{S}.$$
Furthermore there exists a constant $\nu > 0$ such that if $x\in\bord$, then for any $s \in [0,T]$, $d(x+\beta(s,x,e),\bord) \geq \nu$, $\lambda$-a.s. 
\end{enumerate}
These assumptions mean in particular that if $X_{s^-} \in \mathcal{S}$, then $X_s \in \mathcal{S}$ a.s. Moreover if $X_{s^-}$ belongs to the boundary of $\mathcal{S}$, and if there is a jump at time $s$, then $X_s$ is in the interior of $\mathcal{S}$. Let us now state our first main result. 
\begin{thrm} \label{thm:equality}
Under Conditions {\rm \textbf{(A*)-(C)-(D)-(E)}}, with \eqref{eq:f0_int} and \eqref{eq:cond_rho}, the minimal supersolution $Y$ satisfies a.s. 
$$\liminf_{t\to T} Y_t = \xi.$$
\end{thrm}

\subsection{An estimate on $Z$ and $U$: proof of Proposition \ref{prop:sharp_estim_Z_U}}
%--------------------

We have shown that the sequences $Z^n$ and $U^n$ converge in a suitable integrability space on $[0,T-\eps]$ for any $\eps > 0$. Here we want to obtain an estimate on the limits $Z$ and $U$ on the whole time interval $[0,T]$. 

In the sequel let us denote by $\Gamma$ the process
\begin{equation*}
\Gamma_t = \frac{K_{\ell,L,\vartheta}^{\ell}}{T-t} \E \left( \ \int_t^{T} \left[ \left(\frac{1}{qa_s}\right)^{1/q} + (T-s)^{1+1/q} f^0_s \right]^{\ell} ds \bigg| \tri_t\right) 
\end{equation*}
thus Estimate \eqref{eq:a_priori_estimate} becomes:
\begin{equation}\label{eq:a_priori_estimate_2} 
0\leq Y_t \leq \frac{1}{(T-t)^{1+1/q-1/\ell}} \Gamma_t^{1/\ell}.
\end{equation}

\begin{lmm}
Under \eqref{eq:alpha_gamma_2}, $\displaystyle \E  \int_0^T (T-s)^{-1+\eta} \Gamma_s ds < +\infty.$
\end{lmm}
\begin{proof}
Note that
\begin{eqnarray*}
(T-s)^{-1+\eta}\E( \Gamma_s) & = &  K_{\ell,L,\vartheta}^{\ell} (T-s)^{-2+\eta} \int_s^{T} \E \left[ \left(\frac{1}{qa_u}\right)^{1/q} + (T-u)^{1+1/q} f^0_u \right]^{\ell} du \\
& = &  K_{\ell,L,\vartheta}^{\ell} (T-s)^{-2+\eta} \int_0^T \theta_u \ind_{u\geq s} du
\end{eqnarray*}
with 
$$\theta_u =  \E \left[ \left(\frac{1}{qa_u}\right)^{1/q} + (T-u)^{1+1/q} f^0_u \right]^{\ell}.$$
Hence by Fubini's theorem
\begin{eqnarray*}
\E  \int_0^T (T-s)^{-1+\eta} \Gamma_s ds &= & K_{\ell,L,\vartheta}^{\ell}  \int_0^T (T-s)^{-2+\eta} \left(  \int_0^T   \theta_u \ind_{u\geq s} du \right) ds \\
& = & K_{\ell,L,\vartheta}^{\ell}   \int_0^T \theta_u \left(   \int_0^u (T-s)^{-2+\eta} ds \right) du \\
& = &   \frac{K_{\ell,L,\vartheta}^{\ell} }{1-\eta}  \int_0^T(T-u)^{-1+\eta} \theta_u \left( 1  - (1-u/T)^{1-\eta} \right) du \\
& \leq & \frac{K_{\ell,L,\vartheta}^{\ell} }{1-\eta}  \int_0^T(T-u)^{-1+\eta} \theta_u du < +\infty.
\end{eqnarray*}
This achieves the proof of the lemma.
\end{proof} 

Now let us prove the sharper estimates on $Z$ and $U$ given by Proposition \ref{prop:sharp_estim_Z_U}: there exists a constant $C$ independent of $n$ such that the process $(Z^n,U^n)$ satisfies:
$$\E\left[ \int_0^T (T-s)^{\rho} \left( |Z^n_s|^2 + \|U^n_s\|^2_{\bL^2_\lambda} \right) ds  \right]^{\ell/2} \leq C.$$
where the constant $\rho$ is given by \eqref{eq:def_rho}.

\begin{proof}
For the constant $\eta > 0$ of \eqref{eq:alpha_gamma_2}, let us define
$$\delta = \ell-1  +\frac{\ell}{q} + \eta =\frac{\ell}{2} \rho > 0.$$
We define $c(\ell) = \frac{\ell((\ell-1)\wedge 1)}{2}$, $\check{x} = |x|^{-1} x \ind_{x\neq 0}$ and we want to apply It\^o's formula to $(T-t)^\delta |Y^n_t|^\ell$ (see \cite{krus:popi:14}, Corollary 1 and Remark 1). We fix $\eps > 0$ and $\tau = T-\eps$ in the sequel. Hence we have for $0 \leq t \leq \tau$:
\begin{eqnarray} \label{eq:gene_Ito_1}
&&(T-t)^\delta |Y^n_t|^\ell \leq \eps^\delta |Y^n_{T-\eps}|^\ell + \int_t^\tau \delta (T-s)^{\delta-1} |Y^n_s|^\ell ds \\ \nonumber
&&\quad + \ell \int_t^\tau (T-s)^{\delta}|Y^n_s|^{\ell-1} \check{Y}^n_s f(s,Y^n_s,Z^n_s,U^n_s) ds \\ \nonumber
&& \quad  -  \ell \int_t^\tau (T-s)^{\delta}|Y^n_s|^{\ell-1} \check{Y}^n_s Z^n_s dW_s  -  \ell  \int_t^\tau (T-s)^{\delta}|Y^n_{s^-}|^{\ell-1} \check{Y}^n_{s^-} dM^n_s \\ \nonumber
& &\quad  -  \ell \int_t^\tau (T-s)^{\delta} |Y^n_{s^-}|^{\ell-1} \check{Y}^n_{s^-}  \int_E U^n_s(e) \tmu(de,ds)   \\ \nonumber
& &\quad -  \int_{t}^{\tau} (T-s)^{\delta} \int_E \left[ |Y^n_{s^-}+U^n_s(e)|^\ell -|Y^n_{s^-}|^\ell - \ell |Y^n_{s^-}|^{\ell-1} \check{Y}^n_{s^-}  U^n_s(e) \right] \mu(de,ds) \\  \nonumber
&& \quad - \sum_{ t < s \leq \tau} (T-s)^{\delta} \left[ |Y^n_{s^-}+\Delta M^n_s|^\ell - |Y^n_{s^-}|^\ell - \ell |Y^n_{s^-}|^{\ell-1} \check{Y}^n_{s^-}  \Delta M^n_s \right] \\ \nonumber
&&\quad - c(\ell) \int_t^\tau  (T-s)^{\delta} |Y^n_{s}|^{\ell-2} |Z^n_s|^2 \ind_{Y^n_s\neq 0} ds -c(\ell) \int_t^\tau (T-s)^{\delta} |Y^n_{s}|^{\ell-2}  \ind_{Y^n_s\neq 0} d[ M^n ]^c_s .
\end{eqnarray}
The monotonicity Condition \eqref{eq:f_mono} implies that 
\begin{eqnarray*}
&&\int_t^\tau (T-s)^{\delta}|Y^n_s|^{\ell-1} \check{Y}^n_s f(s,Y^n_s,Z^n_s,U^n_s) ds  \leq \int_t^\tau (T-s)^{\delta}|Y^n_s|^{\ell-1} \check{Y}^n_s f(s,0,Z^n_s,U^n_s) ds  
\end{eqnarray*}
and we use the regularity Conditions \eqref{eq:f_lip_z} and \eqref{eq:f_jump_comp} to obtain:
\begin{eqnarray*}
&& \int_t^\tau (T-s)^{\delta}|Y^n_s|^{\ell-1} \check{Y}^n_s f(s,0,Z^n_s,U^n_s) ds   \leq  L \int_t^\tau (T-s)^{\delta}|Y^n_s|^{\ell-1} |Z^n_s| ds  \\
&& \qquad + L \int_t^\tau (T-s)^{\delta}|Y^n_s|^{\ell-1} \|U^n_s\|_{L^2_\lambda}ds + \int_t^\tau (T-s)^{\delta}|Y^n_s|^{\ell-1} f^0_s ds .
\end{eqnarray*}
Young's inequality leads to:
\begin{eqnarray*}
L \int_t^\tau (T-s)^{\delta}|Y^n_s|^{\ell-1} |Z^n_s| ds & \leq &  \frac{L^2\ell^2}{2c(\ell)}\int_t^\tau (T-s)^{\delta}|Y^n_s|^{\ell} ds \\
&+ & \frac{c(\ell)}{2} \int_t^\tau  (T-s)^{\delta} |Y^n_{s}|^{\ell-2} |Z^n_s|^2 \ind_{Y^n_s\neq 0} ds,
\end{eqnarray*}
\begin{eqnarray*}
L \int_t^\tau (T-s)^{\delta}|Y^n_s|^{\ell-1}  \|U^n_s\|_{L^2_\lambda}ds & \leq &  \frac{L^2\ell^2}{2c(\ell)}\int_t^\tau (T-s)^{\delta}|Y^n_s|^{\ell} ds \\
&+ & \frac{c(\ell)}{2} \int_t^\tau  (T-s)^{\delta} |Y^n_{s}|^{\ell-2}  \|U^n_s\|^2_{\bL^2_\lambda} \ind_{Y^n_s\neq 0} ds
\end{eqnarray*}
and
\begin{eqnarray*}
\int_t^\tau (T-s)^{\delta}|Y^n_s|^{\ell-1} f^0_s ds & \leq & (\ell-1) \int_t^\tau (T-s)^{\delta}|Y^n_s|^{\ell} ds +  \int_t^\tau(T-s)^{\ell (1+1/q)}|f^0_s|^{\ell}ds .
\end{eqnarray*}
Finally all local martingales involved above in \eqref{eq:gene_Ito_1} are true martingales. Hence taking the expectation and using the convexity of $x\mapsto |x|^\ell$ we have:
\begin{eqnarray} \label{eq:gene_Ito_2}
&& \sup_{t\in [0,\tau]}\E \left[ (T-t)^{\delta} |Y^n_{t}|^{\ell} \right] + \frac{c(\ell)}{2} \E \int_t^\tau  (T-s)^{\delta} |Y^n_{s}|^{\ell-2} |Z^n_s|^2 \ind_{Y^n_s\neq 0} ds \\ \nonumber
&&\quad \leq \eps^\delta |Y^n_{T-\eps}|^\ell + \E \int_t^\tau \delta (T-s)^{\delta-1} |Y^n_s|^\ell ds \\ \nonumber
&& \qquad + \ell \left(2 \frac{L^2\ell^2}{2c(\ell)} + (\ell-1) \right) \E \int_t^\tau (T-s)^{\delta}|Y^n_s|^{\ell} ds \\ \nonumber
&& \qquad +\ell \E \int_t^\tau(T-s)^{\ell (1+1/q)}|f^0_s|^{\ell}ds \\ \nonumber
& &\qquad - \E \int_{t}^{\tau} (T-s)^{\delta} \int_E \left[ |Y^n_{s^-}+U^n_s(e)|^\ell -|Y^n_{s^-}|^\ell - \ell |Y^n_{s^-}|^{\ell-1} \check{Y}^n_{s^-}  U^n_s(e) \right] \mu(de,ds) \\  \nonumber
&& \qquad + \frac{c(\ell)}{2} \E \int_t^\tau  (T-s)^{\delta} |Y^n_{s}|^{\ell-2}  \|U^n_s\|^2_{\bL^2_\lambda} \ind_{Y^n_s\neq 0} ds.
\end{eqnarray}
From Lemma 9 in \cite{krus:popi:14}, 
 \begin{eqnarray*}
&& \int_{t}^{\tau} (T-s)^{\delta} \int_E \left[ |Y^n_{s^-}+U^n_s(e)|^\ell -|Y^n_{s^-}|^\ell - \ell |Y^n_{s^-}|^{\ell-1} \check{Y}^n_{s^-}  U^n_s(e) \right] \mu(de,ds) \\
&&  \geq c(\ell) \int_{t}^{\tau} (T-s)^{\delta} \int_E |U^n_s(e)|^2  \left( |Y^n_{s^-}|^2 \vee  |Y^n_{s^-} +U^n_s(e)|^2 \right)^{\ell/2-1} \ind_{|Y^n_{s^-}|\vee |Y^n_{s^-}+ U^n_s(e)| \neq 0} \mu(de,ds).
\end{eqnarray*}
By a localization argument the two following exceptations are the same (see proof of Proposition 3 in \cite{krus:popi:14}):
$$\E\int_{t}^{\tau} (T-s)^{\delta} \int_E |U^n_s(e)|^2  \left( |Y^n_{s^-}|^2 \vee  |Y^n_{s^-} +U^n_s(e)|^2 \right)^{\ell/2-1} \ind_{|Y^n_{s^-}|\vee |Y^n_{s^-}+ U^n_s(e)| \neq 0} \mu(de,ds)$$
$$\E \int_t^\tau  (T-s)^{\delta} |Y^n_{s}|^{\ell-2}  \|U^n_s\|^2_{\bL^2_\lambda} \ind_{Y^n_s\neq 0} ds.$$
Finally we have:
\begin{eqnarray} \label{eq:gene_Ito_3}
&&  \sup_{t\in [0,\tau]}\E \left[ (T-t)^{\delta} |Y^n_{t}|^{\ell} \right] +\frac{c(\ell)}{2} \E \int_0^\tau  (T-s)^{\delta} |Y^n_{s}|^{\ell-2} |Z^n_s|^2 \ind_{Y^n_s\neq 0} ds  \\ \nonumber
&&\quad + \frac{c(\ell)}{2} \E \int_0^\tau  (T-s)^{\delta} |Y^n_{s}|^{\ell-2}  \|U^n_s\|^2_{\bL^2_\lambda} \ind_{Y^n_s\neq 0} ds\\ \nonumber
&&\quad \leq\eps^\delta |Y^n_{T-\eps}|^\ell+ \E \int_0^\tau \delta (T-s)^{\delta-1} |Y^n_s|^\ell ds \\ \nonumber 
&& \qquad + \ell \left(2 \frac{L^2\ell^2}{2c(\ell)} + (\ell-1) \right) \E \int_0^\tau (T-s)^{\delta}|Y^n_s|^{\ell} ds  +\ell \E \int_0^\tau(T-s)^{\ell (1+1/q)}|f^0_s|^{\ell}ds .
\end{eqnarray}
Using \eqref{eq:a_priori_estimate_2}, the first term on the right-hand side can be controlled as follows:
\begin{eqnarray*}
\E \int_0^\tau (T-s)^{\delta-1} |Y^n_s|^\ell ds & \leq &\E \int_0^T (T-s)^{\delta-1} \frac{1}{(T-s)^{\ell+\ell/q-1}} \Gamma_s ds \\
& = &\E \int_0^T (T-s)^{-1+\eta}  \Gamma_s ds < +\infty.
\end{eqnarray*}
The second one satisfies the same estimate:
$$ \E \int_0^\tau (T-s)^{\delta}|Y^n_s|^{\ell} ds  \leq \E \int_0^T (T-s)^{\eta}  \Gamma_s ds < +\infty.$$
And the last term is bounded by Condition \eqref{eq:alpha_gamma}. Therefore we can let $\eps$ go to zero in \eqref{eq:gene_Ito_3} and we can replace every $\tau$ by $T$. To finish the proof, we use the same tricks as in the proof of Proposition 3 in \cite{krus:popi:14}. First we can control the quantity 
$$  \E \left[ \sup_{t\in [0,T]}(T-t)^{\delta} |Y^n_{t}|^{\ell} \right] $$
by the same right-hand side (up to some multiplicative constant). Then if $\ell \geq 2$, we use \eqref{eq:gene_Ito_1} with $\ell=2$ and the result follows immediately. If $1 < \ell < 2$, the conclusion is more tricky. Let us define
$\displaystyle \zeta = \sup_{t\in[0,T]} (T-t)^{\delta/\ell}Y^n_{t}$ and:
\begin{eqnarray} \nonumber
&& \E \left( \int_0^T (T-s)^{2\delta /\ell}|Z^n_s|^2 ds \right)^{\ell/2}  =  \E \left( \int_0^T(T-s)^{2\delta /\ell}\ind_{Y^n_s \neq 0} |Z^n_s|^2 ds\right)^{\ell/2} \\ \nonumber
&&\quad = \E \left( \int_0^T (T-s)^{2\delta /\ell}\left( Y^n_{s} \right)^{2-\ell}\left( Y^n_{s} \right)^{\ell-2}\ind_{Y^n_s \neq 0}  |Z^n_s|^2ds \right)^{\ell/2} \\ \nonumber
&& \quad \leq \E \left[ \zeta^{(2-\ell)\ell/2} \left(\int_0^T(T-s)^{\delta} \left( Y^n_{s} \right)^{\ell-2}\ind_{Y^n_s \neq 0}  |Z^n_s|^2 ds \right)^{\ell/2}\right]\\ \nonumber
&& \quad \leq \left\{\E \left[\zeta^{\ell}\right]\right\}^{(2-\ell)/2}  \left\{\E \int_0^T (T-s)^{\delta} \left( Y^n_{s} \right)^{\ell-2}\ind_{Y^n_s \neq 0}  |Z^n_s|^2 ds \right\}^{\ell/2} \\ \nonumber
&& \quad  \leq \frac{2-\ell}{2} \E \left[\zeta^{\ell}\right] + \frac{\ell}{2} \E \int_0^T(T-s)^{\delta} \left( Y^n_{s} \right)^{\ell-2}\ind_{Y^n_s \neq 0}  |Z^n_s|^2 ds < +\infty.
\end{eqnarray}
where we have used H\"older's and Young's inequality with $ \frac{2-\ell}{2} + \frac{\ell}{2}=1$. The same holds for $U^n$. Therefore since
$$2\delta /\ell =  2 \left(1 -\frac{1}{\ell} \right) +\frac{2}{q} + \frac{2\eta}{\ell},$$
we obtain the same desired result.
\end{proof}

Note that from the proof we also could derive an estimate on $M$. But we will not need it in the rest of the paper. 
\begin{rmrk}
If $f(y)=-y|y|^q$, we can take $\ell=1$ and $\eta=0$, in other words $\alpha =2/q$. The constant $C$ is explicitely given by: $C = 16\left(\frac{1}{q}\right)^{2/q}$. The proof is a direct modification of the proof of Proposition 10 in \cite{popi:06}. 
\end{rmrk}

\subsection{Proof of Theorem \ref{thm:equality}} \label{sect:equality}
%-------------------

In order to prove this theorem we follow the same procedure as in \cite{popi:06}. We consider $(Y^n,Z^n,U^n,M^n)$ the solution of the BSDE \eqref{eq:trunc_BSDE} with terminal condition $\xi \wedge n$ and generator $f_n$. Let $\phi$ be a non negative function in $C_b^2(\R)$, the set of bounded smooth functions of class $C^2$, with bounded derivatives. We compute It\^o's formula to the process $Y^n\phi(X)$ between $0$ and $t$. 
\begin{eqnarray*}
Y^n_t\phi(X_t) & =& Y^n_0 \phi(X_0)+\int_0^t Y^n_{s_-} d\phi(X_s) +\int_0^t
\phi(X_{s_-}) dY^n_s+\langle Y^n,\phi(X)\rangle_t \\
& =& Y^n_0\phi(X_0)- \int_0^t \phi(X_{s_-}) f_n(s,Y^n_s,Z^n_s,U^n_s) ds +\int_0^t Y^n_{s_-} \ope\phi(s,X_s) ds \\
&& +\int_0^t\int_E Y^n_{s_-} \left(\phi(X_s)-\phi(X_{s_-})-\nabla\phi(X_{s_-})\beta(s,X_{s_-},e)\right)\mu(ds,de)\\
&& + \int_0^t Y^n_{s_-}\nabla\phi(X_s)\sigma(s,X_s)dW_s+\int_0^t \phi(X_{s_-}) Z^n_s
dW_s + \int_0^t \phi(X_{s_-}) dM^n_s\\
&& + \int_0^t\int_E \phi(X_{s_-})U^n_s(e)\tilde{\mu}(de,ds)  + \int_0^t\int_E Y^n_{s_-}\nabla\phi(X_s)\beta(s,X_{s_-},e)\tilde{\mu}(de,ds)\\
&& +\int_0^t\nabla\phi(X_s)\sigma(s,X_s)Z^n_s ds +\int_0^t\int_E (\phi(X_s)-\phi(X_{s_-}))U^n_s(e)\mu(ds,de).
\end{eqnarray*}
The operators $\ope$ and $\mathcal{I}$ are defined on $C^2(\R)$ by:
$$\ope \phi (t,x) = \nabla\phi(x)b(t,x)+\frac{1}{2} \tr( D^2\phi(x)(\sigma\sigma^*)(t,x))$$
and
\begin{equation}\label{eq:non_local_gene}
\mathcal{I} (t,x,\phi)= \int_E [\phi(x+\beta(t,x,e))-\phi(x) - (\nabla \phi)(x)\beta(t,x,e)] \lambda(de).
\end{equation}
Since $(Y^n,Z^n,U^n,M^n)$ belongs to $\bS^2(0,T)$, since $X$ is in $\bH^2(0,T)$, and since $\phi$ and the derivatives of $\phi$ are supposed to be bounded, we can take the expectation of these terms:
\begin{eqnarray} \label{eq:cont_test_funct_1}
&& \E[Y^n_t\phi(X_t)] = \E[Y^n_0\phi(X_0)] - \E\left[\int_0^t \phi(X_{s_-}) f_n(s,Y^n_s,Z^n_s,U^n_s) ds\right]\\ \nonumber
&& \qquad +  \E\left[\int_0^t Y^n_{s_-} \ope\phi(s,X_s) ds\right] + \E\left[\int_0^t Y^n_{s_-} \mathcal{I} (s,X_{s^-},\phi)ds\right]\\ \nonumber
&& \qquad + \E\left[\int_0^t\nabla\phi(X_s)\sigma(s,X_s)Z^n_s ds\right] + \E\left[\int_0^t\int_E
(\phi(X_s)-\phi(X_{s_-}))U^n_s(e)\lambda(de)ds\right].
\end{eqnarray}
Recall the main idea of \cite{popi:06}. First we prove that we can pass to the limit on $n$ in \eqref{eq:cont_test_funct_1} and that the limits have suitable integrability conditions on $[0,T] \times \Omega$. Secondly we write \eqref{eq:cont_test_funct_1} between $t$ and $T$ and we pass to the limit when $t$ goes to $T$. 

Now we choose $\phi$ such that the support of $\phi$ is included in $\mathcal{R} = \mathcal{S}^c$. From the Assumptions \textbf{(C1)} and \textbf{(C2)} on $\xi=\Phi(X_T)$, we have for any $n$:
$$\E (Y^n_T \phi(X_T)) \leq \E (\Phi(X_T)\phi(X_T)) < +\infty.$$
Moreover from the a priori estimate \eqref{eq:a_priori_estimate_2}, Assumption \eqref{eq:alpha_gamma} and from the boundedness of $\phi$, for any $t < T$
$$\E (Y^n_t \phi(X_t)) \leq \frac{1}{(T-t)^{1/q+1-1\ell}} \E (\Gamma^\ell_t\phi(X_t)) < +\infty.$$
Now we decompose the quantity with the generator $f_n$ as follows:
\begin{eqnarray} \label{eq:decomp_generator}
&&\E\left[\int_0^t \phi(X_{s_-}) f_n(s,Y^n_s,Z^n_s,U^n_s) ds\right] \\ \nonumber
&& \quad = \E\left[\int_0^t \phi(X_{s_-}) (f(s,Y^n_s,0,0) -f^0_s) ds\right]  +\E\left[\int_0^t \phi(X_{s_-}) (f^0_s\wedge n) ds\right] \\ \nonumber
&&\qquad + \E\left[\int_0^t \phi(X_{s_-}) \left( f(s,Y^n_s,Z^n_s,0) - f(s,Y^n_s,0,0)\right) ds\right] \\ \nonumber
&& \qquad + \E\left[\int_0^t \phi(X_{s_-}) \left( f(s,Y^n_s,Z^n_s,U^n_s) - f(s,Y^n_s,Z^n_s,0)\right) ds\right] \\ \nonumber
&& \quad = \E\left[\int_0^t \phi(X_{s_-}) (f(s,Y^n_s,0,0)-f^0_s) ds\right] +\E\left[\int_0^t \phi(X_{s_-}) (f^0_s\wedge n) ds\right] \\ \nonumber
&&\qquad + \E\left[\int_0^t \phi(X_{s_-}) \zeta^n_s Z^n_s ds\right] + \E\left[\int_0^t \phi(X_{s_-}) \mathcal{U}^n_s ds\right]
\end{eqnarray}
where $\zeta^n_s$ is a $k$-dimensional random vector defined by:
$$\zeta^{i,n}_s = \frac{\left( f(s,Y^n_s,Z^n_s,0) - f(s,Y^n_s,0,0)\right) }{Z^{i,n}_s} \ind_{Z^{i,n}_s\neq 0}$$
and
$$\mathcal{U}^n_s  = f(s,Y^n_s,Z^n_s,U^n_s) - f(s,Y^n_s,Z^n_s,0).$$
From Condition \eqref{eq:f_lip_z}, $|\zeta^n_s|\leq K$. Now we can write \eqref{eq:cont_test_funct_1} as follows:
\begin{eqnarray} \label{eq:cont_test_funct_2}
&& \E[Y^n_t\phi(X_t)] = \E[Y^n_0\phi(X_0)] +\E\left[\int_0^t \phi(X_{s_-}) (f^0_s\wedge n) ds\right]  \\ \nonumber
&& \quad - \E\left[\int_0^t \phi(X_{s_-}) (f(s,Y^n_s,0,0)-f^0_s) ds\right]\\ \nonumber
&& \quad +  \E\left[\int_0^t Y^n_{s_-} \ope\phi(s,X_s) ds\right] + \E\left[\int_0^t Y^n_{s_-} \mathcal{I} (s,X_{s^-},\phi)ds\right]\\ \nonumber
&& \quad + \E\left[\int_0^t \left( \nabla\phi(X_s)\sigma(s,X_s) - \phi(X_{s}) \zeta^n_s\right) Z^n_s ds\right] \\ \nonumber
&& \quad + \E\left[\int_0^t \left[ \int_E \left[(\phi(X_{s})-\phi(X_{s_-}))\right] U^n_s(e)\lambda(de) - \phi(X_{s_-})\mathcal{U}^n_s \right] ds\right].
\end{eqnarray}
Since $\phi$ is bounded and $f^0 \in L^1((0,T)\times \Omega)$ (Condition \eqref{eq:f0_int}):
\begin{equation} \label{eq:control_cont_0}
\E \int_0^T | \phi(X_{s_-})| (f^0_s\wedge n) ds \leq C.
\end{equation}
We use H\"older's and Young's inequalities to obtain:
\begin{eqnarray*}
&& \int_0^T |\left( \nabla\phi(X_s)\sigma(s,X_s) + \phi(X_{s}) \zeta^n_s\right)Z^n_s| ds\\ \nonumber
&& \quad  \leq \left[\int_0^T (T-s)^{\rho} |Z^n_s|^2 ds\right]^{1/2}  \left[ \int_0^T \frac{|\nabla\phi(X_s)\sigma(s,X_s) + \phi(X_{s}) \zeta^n_s|^2}{(T-s)^{\rho}} ds \right]^{1/2}  \\  \nonumber
&&\quad \leq \frac{1}{\ell} \left[\int_0^T (T-s)^{\rho} |Z^n_s|^2 ds\right]^{\frac{\ell}{2}}  + \frac{\ell-1}{\ell} \left[ \int_0^T \frac{|\nabla\phi(X_s)\sigma(s,X_s) + \phi(X_{s}) \zeta^n_s|^2}{(T-s)^{\rho}} ds \right]^{\frac{\ell-1}{2\ell}} .
\end{eqnarray*}
Taking the expectation and thanks to Proposition \ref{prop:sharp_estim_Z_U}, the first term on the right-hand side is bounded. For the second term, by assumption, $\phi$ and $\nabla \phi$ are supposed to be bounded, $\zeta^n_s$ is also bounded, $\sigma$ grows linearly and $X \in \bH^2(0,T)$. Hence if $\rho < 1$ (condition \eqref{eq:cond_rho}), there exists a constant $C$ such that for any $n$
\begin{equation} \label{eq:control_cont_1}
\E \int_0^T |\left( \nabla\phi(X_s)\sigma(s,X_s) + \phi(X_{s}) \zeta^n_s\right)Z^n_s| ds\leq C.
\end{equation}
The same estimate holds for $U^n$. Indeed from \eqref{eq:f_jump_comp}:
\begin{eqnarray*}
&& - \int_E \vartheta(e) |U^n_s(e)| \lambda(de)\leq  \int_E \hat \kappa_s^{n}(e) U^n_s(e) \lambda(de) \\
&& \qquad \leq \mathcal{U}^n_s \leq \int_E \kappa_s^{n}(e) U^n_s(e) \lambda(de) \leq \int_E \vartheta(e) |U^n_s(e)| \lambda(de)
\end{eqnarray*}
if $\kappa^n_s(e)=\kappa_s^{Y^n,Z^n,U^n,0}(e)$ and $\hat \kappa_s^{n}(e) = \kappa_s^{Y^n,Z^n,0,U^n}(e)$. Hence
\begin{eqnarray*}
&&\E \int_0^t \left[ \int_E \left| \left[(\phi(X_{s})-\phi(X_{s_-}))\right] U^n_s(e) \right| \lambda(de) \right] + |\phi(X_{s_-})\mathcal{U}^n_s |ds  \\
&&\quad \leq \E \int_0^T\int_E\left[ |\phi(X_{s})-\phi(X_{s_-}) |+\vartheta(e)| \phi(X_{s_-})| \right]|U^n_s(e)| \lambda(de)ds  \\ \nonumber
&&\quad  \leq  \frac{1}{\ell} \E \left[\int_0^T (T-s)^{\rho} \|U^n_s\|^2_{L^2_\la} ds\right]^{\frac{\ell}{2}}\\ \nonumber
&& \qquad +  \frac{l-1}{l}  \E \left[ \int_0^T \int_E \frac{(|\phi(X_{s^-}+\beta(s,X_{s^-},e))-\phi(X_{s_-})|+\vartheta(e) |\phi(X_{s_-})|)^2}{(T-s)^{\rho}} \lambda(de) ds \right]^{\frac{\ell-1}{2\ell}}, \\ \nonumber
&& \quad \leq C.
\end{eqnarray*}
thus
\begin{equation} \label{eq:control_cont_2}
\E \int_0^T \left( \int_E |\phi(X_{s})-\phi(X_{s_-})| |U^n_s(e)| \lambda(de) +  | \phi(X_{s_-})| |\mathcal{U}^n_s| \right) ds \leq C.
\end{equation}

Now we treat the three terms in \eqref{eq:cont_test_funct_2} containing $Y^n$ and $X$:
\begin{equation*} 
 - \E\left[\int_0^t \phi(X_{s_-}) f(s,Y^n_s,0,0) ds\right]+  \E\left[\int_0^t Y^n_{s_-} \ope\phi(s,X_s) ds\right] + \E\left[\int_0^t Y^n_{s_-} \mathcal{I} (s,X_{s^-},\phi)ds\right].
\end{equation*}
By condition \eqref{eq:f_upper_bound}, the first integral is bounded from below by:
\begin{equation}\label{eq:control_gene_power}
- \E\left[\int_0^t \phi(X_{s_-}) f(s,Y^n_s,0,0) ds\right] \geq \E\left[\int_0^t \phi(X_{s_-}) (a_s) |Y^n_s|^{1+q} ds\right] .
\end{equation}
Now we deal with the terms containing the operators $\ope$ and $\mathcal{I}$. With H\"older's inequality we obtain:
\begin{eqnarray*}
&&  \E \left[ \int_0^T |Y^n_{s_-} \ope (\phi) (s,X_s)| ds \right] \leq  \left[\E  \int_0^T a_s \phi(X_s) (Y^n_{s})^{q+1}  ds \right]^{1/(q+1)}\\
&& \hspace{5cm} \times \left[\E \int_0^T a_s^{-1/q} \phi^{-1/q}(X_s) |\ope (\phi) (s,X_s)|^{(q+1)/q}  ds \right]^{q/(q+1)}.
\end{eqnarray*}
To control the second quantity, we will be more specific about the test-function $\phi$. We will assume that $\phi = \psi^\gamma$ where $\psi$ belongs to $C^\infty_b(\R^d)$ with support in $\cR$ and $\gamma> 2(q+1)/q$. Under this setting, there exists a constant $C$ depending only on $\psi$, $\gamma$, $\sigma$ and $b$ such that 
$$|\ope(\phi)| = |\ope(\psi^\gamma)| \leq C \psi^{\gamma-2}.$$ 
Thus for $\gamma > 2(q+1)/q$
$$\phi^{-1/q}(X_s) |\ope (\phi) (s,X_s)|^{(q+1)/q} \leq C \psi^{-\gamma/q  + (\gamma-2)(q+1)/q}(X_s) = C \psi^{\gamma-2(q+1)/q}(X_s)$$
which is bounded. By condition \eqref{eq:alpha_gamma}, $a^{-1/q}$ is in $L^\ell(\Omega)$. Therefore 
$$\E \int_0^T a_s^{-1/q} \phi^{-1/q}(X_s) |\ope (\phi) (s,X_s)|^{(q+1)/q}  ds \leq C$$
for some constant $C$. Then 
\begin{eqnarray} \label{eq:control_cont_3}
&&  \E \left[ \int_0^T |Y^n_{s_-} \ope (\phi) (s,X_s)| ds \right] \leq C \left[ \E \int_0^T a_s \phi(X_s) (Y^n_{s})^{q+1}  ds \right]^{1/(q+1)}.
\end{eqnarray}

The previous steps were very similar to \cite{popi:06}. Therefore the main difference comes from the term
\begin{equation} \label{eq:jump_term}
\E\left[\int_0^t\int_E Y^n_{s_-} \left( \phi(X_s)-\phi(X_{s_-})-\nabla\phi(X_{s_-})\beta(s,X_{s_-},e)\right)\lambda(de)ds\right].
\end{equation}
In order to control this term, assumptions \textbf{(D)} on the jumps of $X$ and $\mathcal{S}$ will be used. Remember that $\cR = \cS^c$ is open and for any $\eps > 0$ we define
$$\Gamma(\eps):=\{x\in\cR: \ d(x,\bord) \geq \eps \}.$$
$\overline{\Gamma(\eps/2)^c}$ and $\Gamma(\eps)$ are two disjoint closed sets of $\R^d$. By the $C^\infty$ Urysohn lemma, there exists a $C^\infty$ function $\psi$ such that $\psi \in [0,1]$, $\psi\equiv 1$ on $\Gamma(\eps)$ and $\psi\equiv 0$ on $\Gamma(\eps/2)^c$. In particular the support of $\psi$ is included in $\cR$ and since $\bord$ is compact, $\psi$ belongs to $C^\infty_b(\R^d)$. We take $\gamma > 2(q+1)/q$ and we define 
\begin{equation}\label{eq:def_test_func}
\phi = \psi^\gamma.
\end{equation} 
Note that $\phi$ also takes its values in $[0,1]$, $\phi\equiv 1$ on $\Gamma(\eps)$ and $\phi\equiv 0$ on $\Gamma(\eps/2)^c$.

Since $\bord$ is compact and of class $C^1$, then there exists a constant $\eps_0  >0$ such that for every $y\in \cR \cap \Gamma(\eps_0)^c$, there exists a unique $z\in\bord$ such that $d(y,\bord)=\|y-z\|$ (see \cite{gilb:trud:01}, section 14.6). 
\begin{lmm} \label{lem:control_jumps}
Under the above assumptions, let us choose $\eps_1 < \eps_0$ such that $(1+K_\beta)\eps_1 < \nu$ ($K_\beta$ is the Lipschitz constant of $\beta$ w.r.t. $x$, condition \eqref{eq:lipsch_cond_jump}). We have for any $0 < \eps < \eps_1$:
$$\psi(X_{s^-})=0\Rightarrow\psi(X_s)=0.$$
Moreover
$$\frac{\psi(X_s)}{\psi(X_{s^-})} = \psi(X_s)  \ind_{\Gamma(\eps)}(X_{s^-}).$$
\end{lmm}
\begin{proof}
We consider the case where $X_{s_-}\notin supp(\psi)$, that is $\psi(X_{s_-})=0$. Thus $X_{s^-}$ is in $\mathcal{S}$ or $X_{s^-}$ is in $\cR$ but $d(X_{s^-},\bord) < \eps$. 
\begin{enumerate}
\item If $X_{s_-}\in\mathcal{S}$, then $X_s\in\mathcal{S}$, hence $\psi(X_s)=0$.
\item Let $z\in\cR$ with $d(z,\bord) < \eps$ and $x\in \bord$ such that $d(z,\mathcal{S})=\|z-x\|$. Let us prove that $z+\beta(s,z,e)\in\mathcal{S}$ by contradiction. Assume that $z+\beta(s,z,e)\notin\mathcal{S}$ and consider the following convex combination:
$$z_t:=(1-t)(z+\beta(s,z,e))+t(x+\beta(s,x,e)).$$
Now since $\beta$ is Lipschitz continuous w.r.t. $x$:
\begin{eqnarray*}
\|z_t-(x+\beta(s,x,e))\| &=& (1-t)\| z+\beta(s,z,e)-x-\beta(s,x,e)\|\\
& \leq & (1-t)(1+K_\beta) \|z-x\| \leq (1-t)(1+K_\beta)\eps \\
& \leq & (1+K_\beta)\eps < \nu.
\end{eqnarray*}
Since $x \in \bord$, $x+\beta(s,x,e)\in \mathcal{S}$. But $z+\beta(s,z,e)\notin \mathcal{S}$. Thus by continuity there exists $t_0 \in (0,1)$ such that 
$$z_{t_0}:=(1-t_0)(z+\beta(s,z,e))+t_0(x+\beta(s,x,e))\in\bord.$$
Thus we have obtained $x \in \bord$ and $z_{t_0}\in \bord$ such that 
$$\|z_{t_0}-(x+\beta(s,x,e))\| < \nu \Rightarrow d(x+\beta(x,s,e),\bord) < \nu.$$
This leads to a contradiction. So we deduce $z+\beta(s,z,e)\in\mathcal{S}$.

Hence if $X_{s_-}\in \cR$ with $d(X_{s^-},\bord) < \eps$, $X_s \in \mathcal{S}$ and $\psi(X_s) = 0$.
\end{enumerate}
Now consider the quotient 
$$\frac{\psi(X_s)}{\psi(X_{s^-})} = \frac{\psi(X_s)}{\psi(X_{s^-})} \ind_{\supp(\psi)}(X_{s^-}).$$
The first part of the proof shows that for any $\eps < \eps_1$, we have:
$$\frac{\psi(X_s)}{\psi(X_{s^-})} = \psi(X_s)  \ind_{\Gamma(\eps)}(X_{s^-}).$$
Indeed if $X_{s^-}$ is in $\supp(\psi) \cap \Gamma(\eps)^c$, then $X_s \in \mathcal{S}$, and thus the quotient is null.
\end{proof}

Now we can deal with the term given by \eqref{eq:jump_term}. By H\"older's inequality we obtain:
\begin{eqnarray*}
&&\E \left[\int_0^t Y^n_{s_-} |\mathcal{I}(s,X_{s^-},\phi)|ds\right] \leq \left[ \E \int_0^t a_s \phi(X_{s^-})(Y^n_s)^{q+1}ds\right]^{\frac{1}{q+1}} \\
&& \qquad \times \left[ \E \int_0^t a_s^{-1/q}\int_E \frac{\left|\phi(X_s)-\phi(X_{s_-})-\nabla \phi(X_{s_-})\beta(s,X_{s_-},e)\right|^{\frac{q+1}{q}}}{\phi(X_{s_-})^{1/q}}\lambda(de)ds.\right]^{\frac{q}{q+1}}
\end{eqnarray*}
Since $\phi = \psi^{\gamma}$, the last integral is controlled by:
\begin{eqnarray*}
&& \psi(X_{s_-})^{-\gamma/q}\left|\psi^\gamma(X_s)-\psi^\gamma(X_{s_-})-\nabla(\psi^\gamma)(X_{s_-})\beta(s,X_{s_-},e)\right|^{\frac{q+1}{q}} \\
&&\qquad  \leq C_q \phi(X_s) \left(  \frac{\psi(X_s)}{\psi(X_{s_-})}\right)^{\gamma/q} + C_q \phi(X_{s^-})  \\
&& \qquad \qquad + C_q \psi^{\gamma-(q+1)/q}(X_{s^-}) |\nabla \psi(X_{s^-}) \beta(s,X_{s_-},e)|.
\end{eqnarray*}
But with Lemma \ref{lem:control_jumps} we obtain:
\begin{eqnarray*}
&& \psi(X_{s_-})^{-\gamma/q}\left|\psi^\gamma(X_s)-\psi^\gamma(X_{s_-})-\nabla(\psi^\gamma)(X_{s_-})\beta(s,X_{s_-},e)\right|^{\frac{q+1}{q}} \\
&&  \leq C_q\left[  \psi^{\frac{\gamma (q+1)}{q}}(X_s)   \ind_{\Gamma(\delta)}(X_{s^-}) + \psi^\gamma(X_{s^-})  +\psi^{\gamma-\frac{q+1}{q}}(X_{s^-}) |\nabla \psi(X_{s^-}) \beta(s,X_{s_-},e)|\right].
\end{eqnarray*}
From the assumption on $\psi$ and since $\gamma > 2(q+1)/q$, with condition \eqref{eq:alpha_gamma} there exists a constant $C$ independent on $n$ such that:
\begin{equation}\label{eq:control_cont_4}
\E \left[\int_0^T Y^n_{s_-}|\mathcal{I}(s,X_{s^-},\phi)|ds\right]  \leq C \left[ \E \int_0^T a_s \phi(X_{s^-})(Y^n_s)^{q+1}ds\right]^{\frac{1}{q+1}} .
\end{equation}
Let us summarize what we obtained. For any $\eps$ small enough, any function $\phi=\psi^\gamma$ with $\gamma > 2(q+1)/q$, from \eqref{eq:cont_test_funct_2} and using \eqref{eq:control_cont_0}, \eqref{eq:control_cont_1}, \eqref{eq:control_cont_2}, \eqref{eq:control_cont_3}, \eqref{eq:control_cont_4} we deduce that there exists a constant $C$ independent of $n$ such that 
\begin{equation} \label{eq:control_cont_main_term}
0\leq \E \int_0^T a_s \phi(X_{s^-})|Y^n_s|^{1+q} ds \leq - \E \int_0^T \phi(X_{s^-})f(s,Y^n_s,0,0) ds \leq C < +\infty.
\end{equation}
Moreover all these estimates show that we can pass to the limit in \eqref{eq:cont_test_funct_2} (see details in \cite{popi:06}) and we have:
\begin{eqnarray} \label{eq:cont_test_funct_3}
&& \E[Y_T\phi(X_T)] = \E[Y_t\phi(X_t)] +\E\left[\int_t^T \phi(X_{s_-}) f^0_s ds\right]  \\ \nonumber
&& \quad - \E\left[\int_t^T \phi(X_{s_-}) f(s,Y_s,0,0) ds\right]\\ \nonumber
&& \quad +  \E\left[\int_t^T Y_{s_-} \ope\phi(s,X_s) ds\right] + \E\left[\int_t^T Y_{s_-} \mathcal{I} (s,X_{s^-},\phi)ds\right]\\ \nonumber
&& \quad + \E\left[\int_t^T \left( \nabla\phi(X_s)\sigma(s,X_s) + \phi(X_{s}) \delta_s\right) Z_s ds\right] \\ \nonumber
&& \quad + \E\left[\int_t^T \left[\int_E (\phi(X_{s})-\phi(X_{s_-}))U_s(e)\lambda(de)+\mathcal{U}_s\phi(X_{s_-}) \right] ds\right].
\end{eqnarray}
Estimate \eqref{eq:control_cont_main_term} also holds with $Y$, and once again from \eqref{eq:control_cont_0}, \eqref{eq:control_cont_1}, \eqref{eq:control_cont_2}, \eqref{eq:control_cont_3} and\eqref{eq:control_cont_4}, we can let $t$ go to $T$ in \eqref{eq:cont_test_funct_3} in order to have:
$$\E\left[ (\liminf_{t \to T} Y_t) \phi(X_T)\right] \leq  \lim_{t\to T} \E[Y_t\phi(X_t)]= \E[\xi \phi(X_T)] .$$
Recall that the function $\phi$ is equal to one on $\Gamma(\eps)$, and $\eps$ can be as small as we want, and we already know that $\liminf_{t \to T} Y_t \geq \xi$ a.s. This last inequality shows that in fact a.s.
$$\liminf_{t \to T} Y_t = \xi.$$
This achieves the proof of Theorem \ref{thm:equality}.

The proof of Theorem \ref{thm:equality} shows that the limit of $Y_t$ exists in mean in the following sense: for smooth function $\phi$
$$\lim_{t \to T} \E (Y_t \phi(X_t)) =  \left\{ \begin{array}{ll}
\E (\xi \phi(X_T)) & \mbox{if  } \supp(\phi) \cap \cS = \emptyset, \\
+ \infty & \mbox{if  } \E ( \phi(X_T) \ind_\cS) > 0.
\end{array} \right.$$

\section*{Conclusion}
%-------------

To finish this paper, we gather together the theorems \ref{thm:exists_limit} and \ref{thm:equality}: under the conditions {\rm \textbf{(A*)-(B)-(C)-(D)-(E)}}, with the assumptions \eqref{eq:f0_int} and \eqref{eq:cond_rho}, if the filtration $\bF$ is left-continuous at time $T$ and if one of the next three cases holds:
\begin{itemize}
\item $f$ does not depend on $u$ or $\pi(t,0,u)=f(t,0,0,u)-f(t,0,0,0) \geq 0$ a.s. for any $t$ and $u$;
\item $\vartheta \in \bL^1_\lambda(E)$ and there exists a constant $\kappa_*>-1$ such that $\kappa^{0,0,u,0}_s(e) \geq \kappa_*$ a.e. for any $(s,u,e)$;
\item $\lambda$ is a finite measure on $E$;
\end{itemize}
then a.s. 
$$\lim_{t\to T} Y_t = \xi.$$
Note that {\bf (C)}, {\bf (D)}, {\bf (E)} depend only on the terminal condition $\xi$ and the forward process $X$. The assumptions \textbf{(A*)}, {\bf (B)}, \eqref{eq:f0_int} and \eqref{eq:cond_rho} are conditions on the generator $f$. They are satisfied in
\begin{itemize}
\item Example \ref{ex:control_ex} with $\lambda$ finite, $\al \in L^\ell((0,T)\times \Omega)$ for $\ell < 2$, $q > 2\ell/(2 - \ell)$ and $\gamma$ belongs to $L^1((0,T)\times \Omega)$ ;
\item Example \ref{ex:toy_ex} with $q > 2$ ;
\item Example \ref{ex:power_sing} with $-1 < \varsigma < q$, $2(1+\varsigma) < q$ and $\varpi < 1$.
\end{itemize}

\vspace{1cm}
\noindent {\bf Acknowledgements.} We would like to thank the anonymous referee for helpful comments and suggestions. 

%%%%%%%%%%%

\bibliographystyle{plain}
\bibliography{biblio_sing_BSDE_jumps}

\end{document}